\documentclass[a4paper]{article}
\usepackage{amsmath,amssymb,amsthm,fullpage,comment,bm}
\usepackage{array,booktabs,ascmac,fullpage,multicol,xcolor}
\usepackage[dvipdfmx]{graphicx}
\newtheorem{thm}{Theorem}[section]

\newtheorem{prop}[thm]{Proposition}
\newtheorem{lem}[thm]{Lemma}

\newtheorem{rem}[thm]{Remark}

\numberwithin{equation}{section}

\allowdisplaybreaks

\title{Initial value space of the four dimensional Painlev\'{e} system with $(A_5+A_1)^{(1)}$ symmetry}
\author{Kazuya Matsugashita \thanks{Graduate School of Science and Engineering, Kindai University, 3-4-1, Kowakae, Higashi-Osaka, Osaka 577-8502, Japan. E-mail: 2444310104w@kindai.ac.jp} \and Takao Suzuki \thanks{Department of Mathematics, Kindai University, 3-4-1, Kowakae, Higashi-Osaka, Osaka 577-8502, Japan. E-mail: suzuki@math.kindai.ac.jp}}
\date{}
\begin{document}
\maketitle

\begin{abstract}
The initial value spaces of the Painlev\'{e} equations are proposed by Okamoto.
They are symplectic manifolds in which the Painlev\'{e} equations are described as polynomial Hamiltonian systems on all coordinates. 
In this article, we construct an initial value space of the four dimensional Painlev\'{e} system with affine Weyl group symmetry of type $(A_5+A_1)^{(1)}$.


\end{abstract}

\section{Introduction}\label{sec:intro}

The Painlev\'{e} equations are nonlinear ordinary differential equations of second order.
Each of them is expressed as a polynomial Hamiltonian system.
Since it has movable poles, we have to consider compactification of its phase space for the analytic continuation.
Then there are infinitely many solutions passing through a point called the accessible singularity in the obtained compact phase space.
Hence it becomes the next problem to give a phase space separating those solutions.
This problem was solved by Okamoto in \cite{Oka} with the aid of the blowing up method.
The obtained space is now called the initial value space (IVS).

Afterward, Takano and his collaborators gave canonical coordinates for the IVSs of the Painlev\'{e} equations explicitly on which the Hamiltonians are polynomials in \cite{MMT,Mat,ST}.
They also showed that each Painlev\'{e} equation is a unique Hamiltonian system which is holomorphic on all coordinates of the IVS.
This property is now called the holomorphy property.
Besides, a relationship between the symmetries for the Painlev\'{e} equations and the IVSs is pointed out in \cite{NTY}.

Sakai classified the discrete Painlev\'{e} equations from viewpoints of symmetries and IVSs in \cite{Sak1}.
According to that, both symmetries and IVSs are labeled in terms of the affine root systems and the continuous Painlev\'{e} equations can be regarded as continuous flows commuting with some discrete Painlev\'{e} equations.
On the other hand, several generalizations of the Painlev\'{e} equations were proposed in \cite{FS,Gar,NY,SY,SuzT2,Tsu}.
Hence we expect that there exists a classification theory of them from viewpoints of symmetries and IVSs.

Some IVSs of the generalized Painlev\'{e} systems have been investigated as follows.
\begin{center}
\begin{tabular}{|c|c|c|c|}\hline
Painlev\'{e} system& differential & difference (symmetry) & holomorphy \\ \hline
Garnier & \cite{Kim}, \cite{SuzM} & \cite{Tak2} & \cite{Sas1} \\ \hline
Noumi-Yamada & \cite{Tah} & \cite{CT} & \cite{SY} \\ \hline
Sasano & & & \cite{SY} \\ \hline
FST & & \cite{Tak1} & \cite{Sas2} \\ \hline
\end{tabular}
\end{center}
Kimura constructed the IVS of the Garnier system in $n$ variables in \cite{Kim} and Suzuki constructed those of the degenerate Garnier systems in $2$ variables in \cite{SuzM}.
In both works, the blowing up method is utilized.
Tahara constructed the IVS of a four dimensional Noumi-Yamada system by using Laurent series expansions of solutions in \cite{Tah}.
Takenawa and Carstea constructed the IVSs of some systems by resolving singularities of symmetries and gave geometric interpretations in terms of the N\'{e}ron-Severi bilattices.
Sasano derived those systems from the holomorphy properties.
In this article, we construct an IVS of the four dimensional FST system.

This article is organized as follows.
In Section \ref{sec:pre}, we recall the definition of initial value space and the explicit formula of the four dimensional FST system with its symmetry.
In Section \ref{sec:compact}, we construct a compact manifold of the phase space of system \eqref{eq:FST} named $\Sigma_\eta$.
In Section \ref{sec:as}, the accessible singularities on $\Sigma_\eta$ are determined.
As a result, two types of the accessible singularities $C_j^{(kl)}$ and $P_j^{(kl)}$ are obtained.
In Section \ref{sec:resolve_C}, we resolve $C_j^{(kl)}$ in two ways, the blowing up method and utilizing formal power series.
In Section \ref{sec:resolve_P}, we resolve $P_j^{(kl)}$ by using formal power series and the symmetry.
In Section \ref{sec:main}, the main result of this article is given.
Some proofs of lemmas are placed in Appendix \ref{sec:appendix} and \ref{sec:appendixB}.

\section{Preliminary}\label{sec:pre}

\subsection{Initial value space}

We prepare some terminologies concerning the initial value space following the previous work \cite{Kim,Oka}.
Let $\mathcal{Q}_0=(\mathbb{C}^m\times B,\pi,B)$ be a trivial fiber space with a total space $\mathbb{C}^m\times B$, a base space $B\subset\mathbb{C}^n$ and a projection $\pi$. 
For each $t_0\in B$, an inverse image $\pi^{-1}(t_0)$ is called a fiber of $\mathcal{Q}_0$.
We consider a Pfaff system in $\mathcal{Q}_0$
\begin{equation}\label{eq:pfaff}
	dx_i - \sum_{j=1}^{n}X_{ij}dt_j = 0\quad (i=1,\ldots,m),
\end{equation}
where $X_{ij}\in\mathbb{C}(t_1,\ldots,t_n)[x_1,\ldots,x_m]$.
We always assume that the base space $B$ does not contain any pole of $X_{ij}$.
We also assume that any solutions to system \eqref{eq:pfaff} can be continued meromorphically to the whole $B$.
Then system \eqref{eq:pfaff} defines a foliation $\mathcal{F}_0$ whose leaves are in one-to-one correspondence with solutions to system \eqref{eq:pfaff}.

We call a leaf of $\mathcal{F}_0$ which has at most one intersection with any fiber of $\mathcal{Q}_0$ a transversal leaf.
We also call a leaf which is entirely contained in a fiber of $\mathcal{Q}_0$ a vertical leaf. 
If a foliation $\mathcal{F}_0$ satisfies the following two conditions, then it is said to be uniform.
\begin{enumerate}
\item[(i)]
All leaves of $\mathcal{F}_0$ are transversal leaves.
\item[(ii)]
For any path $\gamma$ in $B$, if there exists a point which is an intersection of a leaf and the fiber over the starting point of $\gamma$, then the path $\gamma$ can be lifted in the leaf.
\end{enumerate}
Since system \eqref{eq:pfaff} is holomorphic, all leaves of $\mathcal{F}_0$ are transversal leaves.
However, the foliation $\mathcal{F}_0$ is not uniform in general because $\mathbb{C}^m$ is not compact.
Thus we consider a fiber space $\mathcal{Q}=(E,\pi,B)$ satisfying the following three conditions.
\begin{enumerate}
\item[(i)]
The fiber space $\mathcal{Q}_0$ is a fiber subspace of $\mathcal{Q}$.
\item[(ii)]
The foliation $\mathcal{F}_0$ is extended to an uniform foliation $\mathcal{F}$ in the total space $E$.
\item[(iii)]
All leaves of $\mathcal{F}$ intersect with $\mathbb{C}^m\times B$.
\end{enumerate}
Then the fiber space $\mathcal{Q}$ is said to be a P-fibration for system \eqref{eq:pfaff} and a fiber of $\mathcal{Q}$ is said to be a initial value space of system \eqref{eq:pfaff}.

In order to construct a P-fibration for system \eqref{eq:pfaff}, we consider a compactification of $\mathbb{C}^m$, which is named $\Sigma$.
Then there exists a subspace of $\left(\Sigma\setminus\mathbb{C}^m\right)\times B$ whose point is contained in the closure of a leaf of $\mathcal{F}_0$.
This subspace is called accessible singularity.
The condition (ii) for uniform is often not satisfied at the accessible singularity.
In such a case, we have to resolve the singularity with the aid of the blowing up method.
 
\subsection{Four dimensional FST system}

The four dimensional FST system is derived from a similarity reduction of the Drinfeld-Sokolov hierarchy (\cite{FS}) or the UC hierarchy (\cite{Tsu}) and is expressed as the polynomial Hamiltonian system
\begin{equation}\label{eq:FST}\begin{split}
	&dq_i = \frac{\partial H}{\partial p_i}dt,\quad
	dp_i = -\frac{\partial H}{\partial q_i}dt\quad(i=1,2),\\
	&H = H_{\rm VI}\left[\alpha_3+\alpha_5-\eta,\alpha_0,\alpha_2+\alpha_4,\alpha_1\eta;q_1,p_1\right] + H_{\rm VI}\left[\alpha_1+\alpha_5-\eta,\alpha_0+\alpha_2,\alpha_4,\alpha_3\eta;q_2,p_2\right]\\
	&\qquad+ \frac{(q_1-1)(q_2-t)\left\{(q_1p_1+\alpha_1)p_2+(q_2p_2+\alpha_3)p_1\right\}}{t(t-1)},
\end{split}\end{equation}
where
\[
	H_{\rm{VI}}[a,b,c,d;q,p]
	= \frac{q(q-1)(q-t)p^2-\{a(q-1)(q-t)+bq(q-t)+(c-1)q(q-1)\}p+dq}{t(t-1)}
\]
and
\[
	\alpha_0 + \alpha_1 + \alpha_2 + \alpha_3 + \alpha_4 + \alpha_5 = 1.
\]
Here the independent variable $t$ is defined in $B=\mathbb{C}\setminus\{0,1\}$.
The symmetry of system \eqref{eq:FST} was first given in \cite{Tsu}.
Afterward, it was shown that the group of the symmetry is isomorphic to an extended affine Weyl group of type $(A_5+A_1)^{(1)}$ in \cite{MS}.
According to that, this group is generated by the birational transformations $r_0,\ldots,r_5,r'_0,\pi_1,\pi_2,\rho$.
Their actions on the parameters are given as follows.
\begin{align*}\begin{array}{|c|c|c|c|c|c|c|c|c|}\hline
	 & \alpha_0 & \alpha_1 & \alpha_2 & \alpha_3 & \alpha_4 & \alpha_5 & \eta \\ \hline
	r_0 & -\alpha_0 & \alpha_1+\alpha_0 & \alpha_2 & \alpha_3 & \alpha_4 & \alpha_5+\alpha_0 & \eta+\alpha_0 \\ \hline
	r_1 & \alpha_0+\alpha_1 & -\alpha_1 & \alpha_2+\alpha_1 & \alpha_3 & \alpha_4 & \alpha_5 & \eta-\alpha_1 \\ \hline
	r_2 & \alpha_0 & \alpha_1+\alpha_2 & -\alpha_2 & \alpha_3+\alpha_2 & \alpha_4 & \alpha_5 & \eta+\alpha_2  \\ \hline
	r_3 & \alpha_0 & \alpha_1 & \alpha_2+\alpha_3 & -\alpha_3 & \alpha_4+\alpha_3 & \alpha_5 & \eta-\alpha_3 \\ \hline
	r_4 & \alpha_0 & \alpha_1 & \alpha_2 & \alpha_3+\alpha_4 & -\alpha_4 & \alpha_5+\alpha_4 & \eta+\alpha_4 \\ \hline
	r_5 & \alpha_0+\alpha_5 & \alpha_1 & \alpha_2 & \alpha_3 & \alpha_4+\alpha_5 & -\alpha_5 & \eta-\alpha_5 \\ \hline
	r_0' & \alpha_0 & \alpha_1 & \alpha_2 & \alpha_3 & \alpha_4 & \alpha_5 & -\eta+\alpha_1+\alpha_3+\alpha_5 \\ \hline
	\pi_1 & \alpha_1 & \alpha_2 & \alpha_3 & \alpha_4 & \alpha_5 & \alpha_0 & \eta+\alpha_0+\alpha_2+\alpha_4 \\ \hline
	\pi_2 & \alpha_2 & \alpha_3 & \alpha_4 & \alpha_5 & \alpha_0 & \alpha_1 & \eta \\ \hline
	\rho & \alpha_0 & \alpha_5 & \alpha_4 & \alpha_3 & \alpha_2 & \alpha_1 & -\eta+\alpha_1+\alpha_3+\alpha_5 \\ \hline
\end{array}\end{align*}
Their actions on the dependent and independent variables are given by
\begin{align*}
	&r_0(q_i) = q_i\quad (i=1,2),\quad
	r_0(p_1) = p_1 - \frac{\alpha_0}{q_1-1},\quad
	r_0(p_2) = p_2,\quad
	r_0(t) = t,\\
	&r_1(q_1) = q_1 + \frac{\alpha_1}{p_1},\quad
	r_1(q_2) = q_2,\quad
	r_1(p_i) = p_i\quad (i=1,2),\quad
	r_1(t) = t,\\
	&r_2(q_i) = q_i\quad (i=1,2),\quad
	r_2(p_1) = p_1 - \frac{\alpha_2}{q_1-q_2},\quad
	r_2(p_2) = p_2 + \frac{\alpha_2}{q_1-q_2},\quad
	r_2(t) = t,\\
	&r_3(q_1) = q_1,\quad
	r_3(q_2) = q_2 + \frac{\alpha_3}{p_2},\quad
	r_3(p_i) = p_i\quad (i=1,2),\quad
	r_3(t) = t,\\
	&r_4(q_i) = q_i\quad (i=1,2),\quad
	r_4(p_1) = p_1,\quad
	r_4(p_2) = p_2 - \frac{\alpha_4}{q_2-t},\quad
	r_4(t) = t,\\
	&r_5(q_i) = q_i \frac{q_1p_1+q_2p_2+\eta}{q_1p_1+q_2p_2-\alpha_5+\eta},\quad
	r_5(p_i) = p_i \frac{q_1p_1+q_2p_2-\alpha_5+\eta}{q_1p_1+q_2p_2+\eta}\quad (i=1,2),\quad
	r_5(t) = t,\\
	&r'_0(q_i) = \frac{1}{q_i},\quad
	r'_0(p_i) = -q_i(q_ip_i+\alpha_{2i-1})\quad (i=1,2),\quad
	r'_0(t) = \frac{1}{t},\\
	&\pi_1(q_1) = \frac{(q_1-t)p_1+(q_2-t)p_2+\eta}{(q_1-1)p_1+(q_2-t)p_2+\eta},\quad
	\pi_1(q_2) = \frac{(q_1-t)p_1+(q_2-t)p_2+\eta}{(q_1-1)p_1+(q_2-1)p_2+\eta},\\
	&\pi_1(q_1p_1) = \frac{q_1-q_2}{t-1}\left\{(q_1-1)p_1+(q_2-t)p_2+\eta\right\}-\alpha_2,\\
	&\pi_1(q_2p_2) = \frac{q_2-t}{t-1}\{(q_1-1)p_1+(q_2-1)p_2+\eta\}-\alpha_4,\quad
	\pi_1(t) = t,\\
	&\pi_2(q_1) = \frac{q_2}{q_1},\quad
	\pi_2(q_2) = \frac{t}{q_1},\quad
	\pi_2(p_1) = q_1p_2,\quad
	\pi_2(p_2) = -\frac{q_1(q_1p_1+q_2p_2+\eta)}{t},\quad
	\pi_2(t) = t,\\
	&\rho(q_1) = q_1,\quad
	\rho(q_2) = \frac{tq_1}{q_2},\quad
	\rho(p_1) = \frac{q_1p_1+q_2p_2-\alpha_5+\eta}{q_1},\quad
	\rho(p_2) = -\frac{q_2(q_2p_2+\alpha_3)}{tq_1},\quad
	\rho(t) = t.
\end{align*}
System \eqref{eq:FST} is invariant under their actions.
Moreover, groups $\langle r_0,\ldots,r_5\rangle$ and $\langle r'_0,r'_0\pi_2^{-1}\pi_1^2\rangle$ are isomorphic to the affine Weyl group of type $A_5^{(1)}$ and $A_1^{(1)}$ respectively.

\begin{rem}
It is known that a Schlesinger system enjoys the Painlev\'{e} property (see \cite{IS,Mal,Miwa}).
On the other hand, the coordinates of system \eqref{eq:FST} are rational in the coordinates of a certain Schlesinger system (see \cite{Sak2,Tsu}).
Hence we can see that system \eqref{eq:FST} enjoys the Painlev\'{e} property.
\end{rem}

\section{Compactification of the initial phase space}\label{sec:compact}

In this section, we construct a manifold $\Sigma_{\eta}$ which is a $\mathbb{P}^2$ bundle over $\mathbb{P}^2$ as a compactification of the phase space $\mathbb{C}^4$ of system \eqref{eq:FST} following the previous work \cite{Kim,SuzM}.

Let $\xi=[\xi_0:\xi_1:\xi_2]$ be homogeneous coordinate of $\mathbb{P}^2$ and set
\[
	U_i = \left\{\xi=[\xi_0:\xi_1:\xi_2]\in\mathbb{P}^2\bigm|\xi_i\neq0\right\}\quad(i=0,1,2).
\]
Using the homogeneous coordinate $\xi$, we define affine coordinates as
\[
	(q_1,q_2) = \left(\dfrac{\xi_1}{\xi_0},\dfrac{\xi_2}{\xi_0}\right),\quad
	\left(q_1^{(1)},q_2^{(1)}\right) = \left(\dfrac{\xi_0}{\xi_1},\dfrac{\xi_2}{\xi_1}\right),\quad
	\left(q_1^{(2)},q_2^{(2)}\right) = \left(\dfrac{\xi_1}{\xi_2},\dfrac{\xi_0}{\xi_2}\right).
\]
They imply relations
\[
	(q_1,q_2) = \left(\frac{1}{q_1^{(1)}},\frac{q_2^{(1)}}{q_1^{(1)}}\right) = \left(\frac{q_1^{(2)}}{q_2^{(2)}},\frac{1}{q_2^{(2)}}\right),\quad
	\left(q_1^{(1)},q_2^{(1)}\right) = \left(\frac{1}{q_1},\frac{q_2}{q_1}\right),\quad
	\left(q_1^{(2)},q_2^{(2)}\right) = \left(\frac{q_1}{q_2},\frac{1}{q_2}\right).
\]
We also define affine coordinates of direct products $W_{i0}=U_i\times\mathbb{C}^2\ (i=0,1,2)$ by
\[
	W_{00} : (q_1,q_2,p_1,p_2),\quad
	W_{10} : \left(q_1^{(1)},q_2^{(1)},p_1^{(1)},p_2^{(1)}\right),\quad
	W_{20} : \left(q_1^{(2)},q_2^{(2)},p_1^{(2)},p_2^{(2)}\right)
\]
with relations
\begin{equation}\label{eq:rel_q}\begin{split}
	&\left(q_1^{(1)},q_2^{(1)},p_1^{(1)},p_2^{(1)}\right) = \left(\frac{1}{q_1},\frac{q_2}{q_1},-q_1(q_1p_1+q_2p_2+\eta),q_1p_2\right),\\
	&\left(q_1^{(2)},q_2^{(2)},p_1^{(2)},p_2^{(2)}\right) = \left(\frac{q_1}{q_2},\frac{1}{q_2},q_2p_1,-q_2(q_1p_1+q_2p_2+\eta)\right).
\end{split}\end{equation}
We set a manifold $\Sigma_\eta^0$ as the quotient space of $W_{00}\cup W_{10}\cup W_{20}$ by relation \eqref{eq:rel_q}.
Note that the manifold $\Sigma_\eta^0$ can be regarded as a $\mathbb{C}^2$ bundle over $\mathbb{P}^2$.

Let $\zeta^{(i)}=\left[\zeta_0^{(i)}:\zeta_1^{(i)}:\zeta_2^{(i)}\right]\ (i=0,1,2)$ be homogenous coordinates of $\mathbb{P}^2$ as
\[
	p_i = \frac{\zeta_i^{(0)}}{\zeta_0^{(0)}},\quad
	p_i^{(1)} = \frac{\zeta_i^{(1)}}{\zeta_0^{(1)}},\quad
	p_i^{(2)} = \frac{\zeta_i^{(2)}}{\zeta_0^{(2)}}\quad (i=1,2).
\]
We set
\[
	W_{ij} = U_i\times \left\{\zeta^{(i)}=\left[\zeta_0^{(i)}:\zeta_1^{(i)}:\zeta_2^{(i)}\right]\in\mathbb{P}^2\bigm|\zeta_j^{(i)}\neq0\right\},\quad
	W_i = W_{i0}\cup W_{i1}\cup W_{i2}\quad (i,j=0,1,2).
\]
Note that $W_{ij}\simeq\mathbb{C}^4$ and $W_i\simeq\mathbb{C}^2\times\mathbb{P}^2$.
Then relation \eqref{eq:rel_q} is extended to those
\begin{equation}\label{eq:rel_hom}\begin{split}
	&\begin{pmatrix}\zeta_0^{(1)}\\ \zeta_1^{(1)}\\ \zeta_2^{(1)}\end{pmatrix} = Z_{10}\begin{pmatrix}\zeta_0^{(0)}\\ \zeta_1^{(0)}\\ \zeta_2^{(0)}\end{pmatrix},\quad
	Z_{10} = \begin{pmatrix}\xi_0^2&0&0\\ -\eta\xi_0\xi_1&-\xi_1^2&-\xi_1\xi_2\\ 0&0&\xi_0\xi_1\end{pmatrix}, \\
	&\begin{pmatrix}\zeta_0^{(2)}\\ \zeta_1^{(2)}\\ \zeta_2^{(2)}\end{pmatrix} = Z_{20}\begin{pmatrix}\zeta_0^{(0)}\\ \zeta_1^{(0)}\\ \zeta_2^{(0)}\end{pmatrix},\quad
	Z_{20} = \begin{pmatrix}\xi_0^2&0&0\\ 0&\xi_0\xi_2&0\\ -\eta\xi_0\xi_2&-\xi_1\xi_2&-\xi_2^2\end{pmatrix}
\end{split}\end{equation}
on $W_0\cup W_1\cup W_2$.
We set a manifold $\Sigma_\eta$ as the quotient space of $W_0\cup W_1\cup W_2$ by relation \eqref{eq:rel_hom}.
Note that the manifold $\Sigma_\eta$ can be regarded as a $\mathbb{P}^2$ bundle over $\mathbb{P}^2$.
We can take affine coordinates of $W_{ij}$ as follows.
\begin{equation}\label{eq:hom_affine}\begin{split}
	&W_{00}: (q_1,q_2,p_1,p_2) = \left(\frac{\xi_1}{\xi_0},\frac{\xi_2}{\xi_0},\frac{\zeta_1^{(0)}}{\zeta_0^{(0)}},\frac{\zeta_2^{(0)}}{\zeta_0^{(0)}}\right).\\
	&W_{01}: \left(q_1,q_2,p_1^{(01)},p_2^{(01)}\right) = \left(\frac{\xi_1}{\xi_0},\frac{\xi_2}{\xi_0},\frac{\zeta_0^{(0)}}{\zeta_1^{(0)}},\frac{\zeta_2^{(0)}}{\zeta_1^{(0)}}\right).\\
	&W_{02}: \left(q_1,q_2,p_1^{(02)},p_2^{(02)}\right) = \left(\frac{\xi_1}{\xi_0},\frac{\xi_2}{\xi_0},\frac{\zeta_1^{(0)}}{\zeta_2^{(0)}},\frac{\zeta_0^{(0)}}{\zeta_2^{(0)}}\right).\\
	&W_{10}: \left(q_1^{(1)},q_2^{(1)},p_1^{(1)},p_2^{(1)}\right) = \left(\frac{\xi_0}{\xi_1},\frac{\xi_2}{\xi_1},\frac{\zeta_1^{(1)}}{\zeta_0^{(1)}},\frac{\zeta_2^{(1)}}{\zeta_0^{(1)}}\right).\\
	&W_{11}: \left(q_1^{(1)},q_2^{(1)},p_1^{(11)},p_2^{(11)}\right) = \left(\frac{\xi_0}{\xi_1},\frac{\xi_2}{\xi_1},\frac{\zeta_0^{(1)}}{\zeta_1^{(1)}},\frac{\zeta_2^{(1)}}{\zeta_1^{(1)}}\right).\\
	&W_{12}: \left(q_1^{(1)},q_2^{(1)},p_1^{(12)},p_2^{(12)}\right) = \left(\frac{\xi_0}{\xi_1},\frac{\xi_2}{\xi_1},\frac{\zeta_1^{(1)}}{\zeta_2^{(1)}},\frac{\zeta_0^{(1)}}{\zeta_2^{(1)}}\right).\\
	&W_{20}: \left(q_1^{(2)},q_2^{(2)},p_1^{(2)},p_2^{(2)}\right) = \left(\frac{\xi_1}{\xi_2},\frac{\xi_0}{\xi_2},\frac{\zeta_1^{(2)}}{\zeta_0^{(2)}},\frac{\zeta_2^{(2)}}{\zeta_0^{(2)}}\right).\\
	&W_{21}: \left(q_1^{(2)},q_2^{(2)},p_1^{(21)},p_2^{(21)}\right) = \left(\frac{\xi_1}{\xi_2},\frac{\xi_0}{\xi_2},\frac{\zeta_0^{(2)}}{\zeta_1^{(2)}},\frac{\zeta_2^{(2)}}{\zeta_1^{(2)}}\right).\\
	&W_{22}: \left(q_1^{(2)},q_2^{(2)},p_1^{(22)},p_2^{(22)}\right) = \left(\frac{\xi_1}{\xi_2},\frac{\xi_0}{\xi_2},\frac{\zeta_1^{(2)}}{\zeta_2^{(2)}},\frac{\zeta_0^{(2)}}{\zeta_2^{(2)}}\right).
\end{split}\end{equation}
It implies relations
\begin{equation}\label{eq:rel_p}\begin{split}
	&\left(p_1^{(01)},p_2^{(01)}\right) = \left(\frac{1}{p_1},\frac{p_2}{p_1}\right),\quad
	\left(p_1^{(02)},p_2^{(02)}\right) = \left(\frac{p_1}{p_2},\frac{1}{p_2}\right),\\
	&\left(p_1^{(11)},p_2^{(11)}\right) = \left(\frac{1}{p_1^{(1)}},\frac{p_2^{(1)}}{p_1^{(1)}}\right),\quad
	\left(p_1^{(12)},p_2^{(12)}\right) = \left(\frac{p_1^{(1)}}{p_2^{(1)}},\frac{1}{p_2^{(1)}}\right),\\
	&\left(p_1^{(21)},p_2^{(21)}\right) = \left(\frac{1}{p_1^{(2)}},\frac{p_2^{(2)}}{p_1^{(2)}}\right),\quad
	\left(p_1^{(22)},p_2^{(22)}\right) = \left(\frac{p_1^{(2)}}{p_2^{(2)}},\frac{1}{p_2^{(2)}}\right).
\end{split}\end{equation}

\begin{rem}
The manifold $\Sigma_\eta$ was proposed by Kimura in \cite{Kim} as a higher dimensional extension of the Hirzebruch surface $\Sigma_2$.
He called $\Sigma_\eta$ a Hirzebruch manifold.
Although we see that the manifold $\Sigma_\eta$ is isomorphic to $\mathbb{P}^2\times\mathbb{P}^2$ in \cite{Kim,SuzM}, it may be incorrect.
Recall that the Hirzebruch surface $\Sigma_2$ is obtained from $\mathbb{P}^1\times\mathbb{P}^1$ through certain blowing up and blowing down procedures.
The manifold $\Sigma_\eta$ can be obtained from $\mathbb{P}^2\times\mathbb{P}^2$ in a similar manner.
We do not state its detail here.
\end{rem}

We extend system \eqref{eq:FST} to that on $\Sigma_{\eta}\times B$ (or resp. $\Sigma_{\eta}^0\times B$) and denote it by $\mathcal{H}$ (or resp. $\mathcal{H}^0$).

\begin{prop}\label{prop:Wi0_ham}
The system $\mathcal{H}^0$ is described as the polynomial Hamiltonian system on any affine chart $W_{i0}$.
\end{prop}

\begin{proof}
Let $\Psi_1,\Psi_2$ be transformations which act on the initial coordinates $q_1,q_2,p_1,p_2$ as
\[
	\Psi_1((q_1,q_2,p_1,p_2)) = \left(q_2,tq_1,p_2,\frac{p_1}{t}\right),\quad
	\Psi_2((q_1,q_2,p_1,p_2)) = \left(tq_2,tq_1,\frac{p_2}{t},\frac{p_1}{t}\right).
\]
Then we obtain
\[
	\left(q_1^{(1)},q_2^{(1)},p_1^{(1)},p_2^{(1)}\right) = \Psi_1\pi_2^{-1}((q_1,q_2,p_1,p_2)),\quad
	\left(q_1^{(2)},q_2^{(2)},p_1^{(2)},p_2^{(2)}\right) = \Psi_2\pi_2((q_1,q_2,p_1,p_2)).
\]
Since the transformation $\pi_2$ provides the symmetry of system \eqref{eq:FST}, the system on $W_{10}\times B$ and $W_{20}\times B$ are both polynomial Hamiltonian systems.
Note that both relations in \eqref{eq:rel_q} are symplectic transformations.
\end{proof}

\begin{prop}\label{prop:sol00}
If we assume that
\begin{equation}\label{eq:prop_sol00}
	\alpha_1 - \eta \neq 0,\quad
	\alpha_3 - \eta \neq 0,\quad
	\alpha_1 + \alpha_3 - \eta \neq 0,\quad
	\alpha_1 + \alpha_2 + \alpha_3 - \eta \neq 0,
\end{equation}
any solution to the system $\mathcal{H}^0$ passes through $W_{00}\times B$.
\end{prop}

\begin{proof}
We assume that there exists a solution to the system $\mathcal{H}^0$ satisfying the condition $\xi_0\equiv0$, namely $q_1^{(1)}\equiv0$ or $q_2^{(2)}\equiv0$.

We first consider the condition $q_1^{(1)}\equiv0$.
Then the equation for $q_1^{(1)}$ in $\mathcal{H}^0$ is described as
\[
	t(t-1)dq_1^{(1)} + \left\{\left(q_2^{(1)}\right)^2p_2^{(1)}-q_2^{(1)}p_2^{(1)}+\alpha_3q_2^{(1)}+\alpha_1-\eta\right\}dt = 0.
\]
We set
\[
	Q_1 = \left(q_2^{(1)}\right)^2p_2^{(1)} - q_2^{(1)}p_2^{(1)} + \alpha_3q_2^{(1)} + \alpha_1 - \eta.
\]
Then we obtain $Q_1\equiv0$, which implies $q_2^{(1)}\not\equiv0,1$ due to assumption \eqref{eq:prop_sol00}.
Moreover, the equation for $Q_1$ in $\mathcal{H}^0$ is described as
\[
	t\left(q_2^{(1)}-1\right)dQ_1 + (\alpha_1+\alpha_3-\eta)(\alpha_1+\alpha_2+\alpha_3-\eta)q_2^{(1)}dt = 0.
\]
The equality $(\alpha_1+\alpha_3-\eta)(\alpha_1+\alpha_2+\alpha_3-\eta)q_2^{(1)}\equiv0$ contradicts with $q_2^{(1)}\not\equiv0$ and \eqref{eq:prop_sol00}.
Therefore there does not exist any solution satisfying $q_1^{(1)}\equiv0$.

We can prove for $q_2^{(2)}\equiv0$ in a similar manner.
The second condition of \eqref{eq:prop_sol00} is required in this case.
\end{proof}

\begin{rem}
Under the condition $\alpha_1-\eta=0$, the system $\mathcal{H}^0$ on $W_{10}\times B$ admits a particular solution such that $q_1^{(1)}=q_2^{(1)}=0$.
Then the equations for $p_1^{(1)},p_2^{(1)}$ are described as 
\begin{align*}
	&t(1-t)\frac{dp_1^{(1)}}{dt} = \left(p_1^{(1)}\right)^2 + \{-(\alpha_0+\alpha_1+\alpha_5)t+\alpha_0\}p_1^{(1)} + \alpha_1\alpha_5t + \left(p_1^{(1)}-\alpha_5\right)tp_2^{(1)},\\
	&t(1-t)\frac{dp_2^{(1)}}{dt} = t\left(p_2^{(1)}\right)^2 + \{\alpha_2t-(\alpha_1+\alpha_2+\alpha_3)\}p_2^{(1)} + \alpha_1\alpha_3 + \left(p_2^{(1)}-\alpha_3\right)p_1^{(1)},
\end{align*}
whose solutions are given as ratios of the generalized hypergeometric series ${}_3F_2$; see \cite{SuzT1} for its detail.
In a similar manner, we can derive particular solutions for $\alpha_3-\eta=0$, $\alpha_1+\alpha_3-\eta=0$ and $\alpha_1+\alpha_2+\alpha_3-\eta=0$.
\end{rem}

\begin{prop}\label{prop:sing}
The system $\mathcal{H}$ has pole singularities on $D^{0}=\left(\Sigma_\eta\setminus\Sigma_\eta^0\right)\times B$.
\end{prop}

\begin{proof}
We first investigate the singularities of the system $\mathcal{H}$ on $W_0\times B$.
The system $\mathcal{H}$ on $W_{01}\times B$ is described as
\begin{equation}\label{eq:W01_1}\begin{split}
	&t(t-1)p_1^{(01)}dq_1- F_1^{(01)}dt = 0,\\
	&t(t-1)p_1^{(01)}dq_2- F_2^{(01)}dt = 0,\\
	&t(t-1)dp_1^{(01)} - G_1^{(01)}dt = 0,\\
	&t(t-1)p_1^{(01)}dp_2^{(01)} - G_2^{(01)}dt = 0,
\end{split}\end{equation}
where
\begin{align*}
	&F_1^{(01)} = \left[(\alpha_1+\eta)q_1^2 + \alpha_3q_1q_2 + \{(\alpha_0+\alpha_5-\eta)t-(\alpha_0+\alpha_1+\eta)\}q_1 - \alpha_3q_2 - (\alpha_5-\eta)t\right]p_1^{(01)}\\
	&\qquad\qquad +(q_1-1)(q_1+q_2)(q_2-t)p_2^{(01)}+2q_1(q_1-1)(q_1-t),\\
	&F_2^{(01)} = \left[\alpha_1q_1q_2 + (\alpha_3+\eta)q_2^2 - \alpha_1tq_1 - \{(\alpha_3+\alpha_4+\eta-1)t-(\alpha_4+\alpha_5-\eta-1)\}q_2 - (\alpha_5-\eta)t\right]p_1^{(01)}\\
	&\qquad\qquad + 2q_2(q_2-1)(q_2-t)p_2^{(01)}+(q_1-1)(q_1+q_2)(q_2-t),\\
	&G_1^{(01)} = \alpha_1\eta\left(p_1^{(01)}\right)^2 + \{2(\alpha_1+\eta)q_1+(\alpha_1+\alpha_3)q_2+(\alpha_0-\alpha_1+\alpha_5-\eta)t-(\alpha_0+\alpha_1+\eta)\}p_1^{(01)}\\
	&\qquad\qquad +(2q_1+q_2-1)(q_2-t)p_2^{(01)} + 3q_1^2 - 2(t+1)q_1 + t,\\
	&G_2^{(01)} = \eta\left(\alpha_1p_2^{(01)}-\alpha_3\right)\left(p_1^{(01)}\right)^2\\
	&\qquad\qquad + \bigg[\alpha_1(q_2-t)\left(p_2^{(01)}\right)^2 + \{(\alpha_1+2\eta)q_1-(\alpha_3+2\eta)q_2-\alpha_1t+\alpha_2+\alpha_3\}p_2^{(01)}-\alpha_3(q_1-1)\bigg]p_1^{(01)}\\
	&\qquad\qquad + \left\{2q_1q_2-2q_2^2-2tq_1+(t+1)q_2\right\}\left(p_2^{(01)}\right)^2+\left\{2q_1q_2+2q_2^2-(t+1)q_1+2q_2\right\}p_2^{(01)}.
\end{align*}
Let $A$ be a $4\times5$ matrix defined by 
\[
	A = \begin{pmatrix}t(t-1)p_1^{(01)}&0&0&0&-F_1^{(01)}\\ 0&t(t-1)p_1^{(01)}&0&0&-F_2^{(01)}\\ 0&0&t(t-1)&0&-G_1^{(01)}\\ 0&0&0&t(t-1)p_1^{(01)}&-G_2^{(01)}\end{pmatrix}.
\]
Then the singularity of system \eqref{eq:W01_1} is the zero set of minor determinants of $A$ of fourth order. 
It is given by
\[
	\left\{\left(q_1,q_2,p_1^{(01)},p_2^{(01)}\right)\in\mathbb{C}^4\bigm|p_1^{(01)}=0\right\}\times B.
\]
Similarly, the singularity of the system $\mathcal{H}$ on $W_{02}\times B$ is
\[
	\left\{\left(q_1,q_2,p_1^{(02)},p_2^{(02)}\right)\in\mathbb{C}^4\bigm|p_2^{(02)}=0\right\}\times B.
\]
By using $p_1^{(01)}=\frac{\zeta_0^{(0)}}{\zeta_1^{(0)}}$ and $p_2^{(02)}=\frac{\zeta_0^{(0)}}{\zeta_2^{(0)}}$, we can rewrite the singularity on $W_0\times B$ in terms of the homogeneous coordinate \eqref{eq:hom_affine} as
\[
	\left\{\left(\xi,\zeta^{(0)}\right)\in\mathbb{P}^2\times\mathbb{P}^2\bigm|\xi_0\neq0,\ \zeta_0^{(0)}=0\right\}\times B.
\]
Note that the system $\mathcal{H}$ coincides with system \eqref{eq:FST} and has no singularity on $W_{00}\times B$.

We can obtain the singularities of the system $\mathcal{H}$ on $W_1\times B$ and $W_2\times B$ as
\[
	\left\{\left(\xi,\zeta^{(1)}\right)\in\mathbb{P}^2\times\mathbb{P}^2\bigm|\xi_1\neq0,\ \zeta_0^{(1)}=0\right\}\times B
\]
and
\[
	\left\{\left(\xi,\zeta^{(2)}\right)\in\mathbb{P}^2\times\mathbb{P}^2\bigm|\xi_2\neq0,\ \zeta_0^{(2)}=0\right\}\times B
\]
in a similar manner.
Hence the system $\mathcal{H}$ has pole singularities on $\left(\Sigma_\eta\setminus\Sigma_\eta^0\right)\times B$.
\end{proof}

\section{Accessible singularities}\label{sec:as}

In the previous section, we have shown that the system $\mathcal{H}$ has the singularities only on $W_{ij}\times B\ (i=0,1,2,\ j=1,2)$.
There exist accessible singularities on these $W_{ij}\times B$ where leaves corresponding to solutions to $\mathcal{H}$ passing through are transversal leaves.
For example, the accessible singularities on $W_{01}\times B$ are given by the condition $p_1^{(01)}=F_1^{(01)}=F_2^{(01)}=G_2^{(01)}=0$.
It follows from the form of system \eqref{eq:W01_1}.
In this section, we list all of the accessible singularities of $\mathcal{H}$.

We first give the following lemma.

\begin{lem}\label{lem:as_list}
All of the accessible singularities of the system $\mathcal{H}$ are given by
\begin{align*}
	&C_1^{(01)} = \left\{\left(q_1,q_2,p_1^{(01)},p_2^{(01)},t\right)\bigm|q_1-q_2=p_1^{(01)}=p_2^{(01)}+1=0\right\},\\
	&C_2^{(01)} = \left\{\left(q_1,q_2,p_1^{(01)},p_2^{(01)},t\right)\bigm|q_1-1=p_1^{(01)}=p_2^{(01)}=0\right\} ,\\
	&C_3^{(01)} = \left\{\left(q_1,q_2,p_1^{(01)},p_2^{(01)},t\right)\bigm|q_1=q_2=p_1^{(01)}=0\right\},\\
	&P_1^{(01)} = \left\{\left(q_1,q_2,p_1^{(01)},p_2^{(01)},t\right)\bigm|q_1-t=q_2-t=p_1^{(01)}=p_2^{(01)}=0\right\},\\
	&P_2^{(01)} = \left\{\left(q_1,q_2,p_1^{(01)},p_2^{(01)},t\right)\bigm|q_1-t=q_2+t=p_1^{(01)}=p_2^{(01)}=0\right\},\\
	&P_3^{(01)} = \left\{\left(q_1,q_2,p_1^{(01)},p_2^{(01)},t\right)\bigm|q_1=q_2-t=p_1^{(01)}=p_2^{(01)}=0\right\},\\
	&P_4^{(01)} = \left\{\left(q_1,q_2,p_1^{(01)},p_2^{(01)},t\right)\bigm|q_1=q_2-t=p _1^{(01)}=(t-1)p_2^{(01)}-2=0\right\},\\
	&P_5^{(01)} = \left\{\left(q_1,q_2,p_1^{(01)},p_2^{(01)},t\right)\bigm|q_1-1=q_2=p_1^{(01)}=2tp_2^{(01)}+t-1=0\right\},\\
	&P_6^{(01)} = \left\{\left(q_1,q_2,p_1^{(01)},p_2^{(01)},t\right)\bigm|q_1-1=q_2-t=p_1^{(01)}=tp_2^{(01)}+1=0\right\},\\
	&P_7^{(01)} = \left\{\left(q_1,q_2,p_1^{(01)},p_2^{(01)},t\right)\bigm|q_1+1=q_2+t=p_1^{(01)}=tp_2^{(01)}+1=0\right\}
\end{align*}
on $W_{01}\times B$,
\begin{align*}
	&C_1^{(02)} = \left\{\left(q_1,q_2,p_1^{(02)},p_2^{(02)},t\right)\bigm|q_1-q_2=p_1^{(02)}+1=p_2^{(02)}=0\right\},\\
	&C_2^{(02)} = \left\{\left(q_1,q_2,p_1^{(02)},p_2^{(02)},t\right)\bigm|q_2-t=p_1^{(02)}=p_2^{(02)}=0\right\},\\
	&C_3^{(02)} = \left\{\left(q_1,q_2,p_1^{(02)},p_2^{(02)},t\right)\bigm|q_1=q_2=p_2^{(02)}=0\right\},\\
	&P_1^{(02)} = \left\{\left(q_1,q_2,p_1^{(02)},p_2^{(02)},t\right)\bigm|q_1-1=q_2-1=p_1^{(02)}=p_2^{(02)}=0\right\},\\
	&P_2^{(02)} = \left\{\left(q_1,q_2,p_1^{(02)},p_2^{(02)},t\right)\bigm|q_1+1=q_2-1=p_1^{(02)}=p_2^{(02)}=0\right\},\\
	&P_3^{(02)} = \left\{\left(q_1,q_2,p_1^{(02)},p_2^{(02)},t\right)\bigm|q_1-1=q_2=p_1^{(02)}=p_2^{(02)}=0\right\},\\
	&P_4^{(02)} = \left\{\left(q_1,q_2,p_1^{(02)},p_2^{(02)},t\right)\bigm|q_1=q_2-t=2p_1^{(02)}-t+1=p_2^{(02)}=0\right\},\\
	&P_5^{(02)} = \left\{\left(q_1,q_2,p_1^{(02)},p_2^{(02)},t\right)\bigm|q_1-1=q_2=(t-1)p_1^{(02)}+2t=p_2^{(02)}=0\right\},\\
	&P_6^{(02)} = \left\{\left(q_1,q_2,p_1^{(02)},p_2^{(02)},t\right)\bigm|q_1-1=q_2-t=p_1^{(02)}+t=p_2^{(02)}=0\right\},\\
	&P_7^{(02)} = \left\{\left(q_1,q_2,p_1^{(02)},p_2^{(02)},t\right)\bigm|q_1+1=q_2+t=p_1^{(02)}+t=p_2^{(02)}=0\right\}
\end{align*}
on $W_{02}\times B$,
\begin{align*}
	&C_1^{(11)} = \left\{\left(q_1^{(1)},q_2^{(1)},p_1^{(11)},p_2^{(11)},t\right)\bigm|tq_1^{(1)}-q_2^{(1)}=p_1^{(11)}=tp_2^{(11)}+1=0\right\},\\
	&C_2^{(11)} = \left\{\left(q_1^{(1)},q_2^{(1)},p_1^{(11)},p_2^{(11)},t\right)\bigm|q_1^{(1)}-1=p_1^{(11)}=p_2^{(11)}=0\right\},\\
	&C_3^{(11)} = \left\{\left(q_1^{(1)},q_2^{(1)},p_1^{(11)},p_2^{(11)},t\right)\bigm|q_1^{(1)}=q_2^{(1)}=p_1^{(11)}=0\right\},\\
	&P_1^{(11)} = \left\{\left(q_1^{(1)},q_2^{(1)},p_1^{(11)},p_2^{(11)},t\right)\bigm|tq_1^{(1)}-1=q_2^{(1)}-1=p_1^{(11)}=p_2^{(11)}=0\right\},\\
	&P_2^{(11)} = \left\{\left(q_1^{(1)},q_2^{(1)},p_1^{(11)},p_2^{(11)},t\right)\bigm|tq_1^{(1)}-1=q_2^{(1)}+1=p_1^{(11)}=p_2^{(11)}=0\right\},\\
	&P_3^{(11)} = \left\{\left(q_1^{(1)},q_2^{(1)},p_1^{(11)},p_2^{(11)},t\right)\bigm|q_1^{(1)}=q_2^{(1)}-1=p_1^{(11)}=p_2^{(11)}=0\right\},\\
	&P_4^{(11)} = \left\{\left(q_1^{(1)},q_2^{(1)},p_1^{(11)},p_2^{(11)},t\right)\bigm|q_1^{(1)}-1=q_2^{(1)}=p_1^{(11)}=2tp_2^{(11)}-t+1=0\right\},\\
	&P_5^{(11)} = \left\{\left(q_1^{(1)},q_2^{(1)},p_1^{(11)},p_2^{(11)},t\right)\bigm|q_1^{(1)}=q_2^{(1)}-1=p_1^{(11)}=(t-1)p_2^{(11)}+2=0\right\},\\
	&P_6^{(11)} = \left\{\left(q_1^{(1)},q_2^{(1)},p_1^{(11)},p_2^{(11)},t\right)\bigm|q_1^{(1)}-1=q_2^{(1)}-1=p_1^{(11)}=p_2^{(11)}+1=0\right\},\\
	&P_7^{(11)} = \left\{\left(q_1^{(1)},q_2^{(1)},p_1^{(11)},p_2^{(11)},t\right)\bigm|q_1^{(1)}+1=q_2^{(1)}+1=p_1^{(11)}=p_2^{(11)}+1=0\right\}
\end{align*}
on $W_{11}\times B$,
\begin{align*}
	&C_1^{(12)} = \left\{\left(q_1^{(1)},q_2^{(1)},p_1^{(12)},p_2^{(12)},t\right)\bigm|tq_1^{(1)}-q_2^{(1)}=p_1^{(12)}+t=p_2^{(12)}=0\right\},\\
	&C_2^{(12)} = \left\{\left(q_1^{(1)},q_2^{(1)},p_1^{(12)},p_2^{(12)},t\right)\bigm|q_2^{(1)}-1=p_1^{(12)}=p_2^{(12)}=0\right\},\\
	&C_3^{(12)} = \left\{\left(q_1^{(1)},q_2^{(1)},p_1^{(12)},p_2^{(12)},t\right)\bigm|q_1^{(1)}=q_2^{(1)}=p_2^{(12)}=0\right\},\\
	&P_1^{(12)} = \left\{\left(q_1^{(1)},q_2^{(1)},p_1^{(12)},p_2^{(12)},t\right)\bigm|q_1^{(1)}-1=q_2^{(1)}-t=p_1^{(12)}=p_2^{(12)}=0\right\},\\
	&P_2^{(12)} = \left\{\left(q_1^{(1)},q_2^{(1)},p_1^{(12)},p_2^{(12)},t\right)\bigm|q_1^{(1)}+1=q_2^{(1)}-t=p_1^{(12)}=p_2^{(12)}=0\right\},\\
	&P_3^{(12)} = \left\{\left(q_1^{(1)},q_2^{(1)},p_1^{(12)},p_2^{(12)},t\right)\bigm|q_1^{(1)}-1=q_2^{(1)}=p_1^{(12)}=p_2^{(12)}=0\right\},\\
	&P_4^{(12)} = \left\{\left(q_1^{(1)},q_2^{(1)},p_1^{(12)},p_2^{(12)},t\right)\bigm|q_1^{(1)}-1=q_2^{(1)}=(t-1)p_1^{(12)}-2t=p_2^{(12)}=0\right\},\\
	&P_5^{(12)} = \left\{\left(q_1^{(1)},q_2^{(1)},p_1^{(12)},p_2^{(12)},t\right)\bigm|q_1^{(1)}=q_2^{(1)}-1=2p_1^{(12)}+t-1=p_2^{(12)}=0\right\},\\
	&P_6^{(12)} = \left\{\left(q_1^{(1)},q_2^{(1)},p_1^{(12)},p_2^{(12)},t\right)\bigm|q_1^{(1)}-1=q_2^{(1)}-1=p_1^{(12)}+1=p_2^{(12)}=0\right\},\\ 
	&P_7^{(12)} = \left\{\left(q_1^{(1)},q_2^{(1)},p_1^{(12)},p_2^{(12)},t\right)\bigm|q_1^{(1)}+1=q_2^{(1)}+1=p_1^{(12)}+1=p_2^{(12)}=0\right\}
\end{align*}
on $W_{12}\times B$,
\begin{align*}
	&C_1^{(21)} = \left\{\left(q_1^{(2)},q_2^{(2)},p_1^{(21)},p_2^{(21)},t\right)\bigm|q_1^{(2)}-q_2^{(2)}=p_1^{(21)}=p_2^{(21)}+1=0\right\},\\
	&C_2^{(21)} = \left\{\left(q_1^{(2)},q_2^{(2)},p_1^{(21)},p_2^{(21)},t\right)\bigm|q_1^{(2)}-1=p_1^{(21)}=p_2^{(21)}=0\right\},\\
	&C_3^{(21)} = \left\{\left(q_1^{(2)},q_2^{(2)},p_1^{(21)},p_2^{(21)},t\right)\bigm|q_1^{(2)}=q_2^{(2)}=p_1^{(21)}=0\right\},\\
	&P_1^{(21)} = \left\{\left(q_1^{(2)},q_2^{(2)},p_1^{(21)},p_2^{(21)},t\right)\bigm|tq_1^{(2)}-1=tq_2^{(2)}-1=p_1^{(21)}=p_2^{(21)}=0\right\},\\
	&P_2^{(21)} = \left\{\left(q_1^{(2)},q_2^{(2)},p_1^{(21)},p_2^{(21)},t\right)\bigm|tq_1^{(2)}-1=tq_2^{(2)}+1=p_1^{(21)}=p_2^{(21)}=0\right\},\\
	&P_3^{(21)} = \left\{\left(q_1^{(2)},q_2^{(2)},p_1^{(21)},p_2^{(21)},t\right)\bigm|q_1^{(2)}=tq_2^{(2)}-1=p_1^{(21)}=p_2^{(21)}=0\right\},\\
	&P_4^{(21)} = \left\{\left(q_1^{(2)},q_2^{(2)},p_1^{(21)},p_2^{(21)},t\right)\bigm|q_1^{(2)}-1=q_2^{(2)}=p_1^{(21)}=2p_2^{(21)}-t+1=0\right\},\\
	&P_5^{(21)} = \left\{\left(q_1^{(2)},q_2^{(2)},p_1^{(21)},p_2^{(21)},t\right)\bigm|q_1^{(2)}=tq_2^{(2)}-1=p_1^{(21)}=(t-1)p_2^{(21)}+2t=0\right\},\\
	&P_6^{(21)} = \left\{\left(q_1^{(2)},q_2^{(2)},p_1^{(21)},p_2^{(21)},t\right)\bigm|q_1^{(2)}-1=tq_2^{(2)}-1=p_1^{(21)}=p_2^{(21)}+t=0\right\},\\
	&P_7^{(21)} = \left\{\left(q_1^{(2)},q_2^{(2)},p_1^{(21)},p_2^{(21)},t\right)\bigm|q_1^{(2)}+1=tq_2^{(2)}+1=p_1^{(21)}=p_2^{(21)}+t=0\right\}
\end{align*}
on $W_{21}\times B$ and
\begin{align*}
	&C_1^{(22)} = \left\{\left(q_1^{(2)},q_2^{(2)},p_1^{(22)},p_2^{(22)},t\right)\bigm|q_1^{(2)}-q_2^{(2)}=p_1^{(22)}+1=p_2^{(22)}=0\right\},\\
	&C_2^{(22)} = \left\{\left(q_1^{(2)},q_2^{(2)},p_1^{(22)},p_2^{(22)},t\right)\bigm|tq_2^{(2)}-1=p_1^{(22)}=p_2^{(22)}=0\right\},\\
	&C_3^{(22)} = \left\{\left(q_1^{(2)},q_2^{(2)},p_1^{(22)},p_2^{(22)},t\right)\bigm|q_1^{(2)}=q_2^{(2)}=p_2^{(22)}=0\right\},\\
	&P_1^{(22)} = \left\{\left(q_1^{(2)},q_2^{(2)},p_1^{(22)},p_2^{(22)},t\right)\bigm|q_1^{(2)}-1=q_2^{(2)}-1=p_1^{(22)}=p_2^{(22)}=0\right\},\\
	&P_2^{(22)} = \left\{\left(q_1^{(2)},q_2^{(2)},p_1^{(22)},p_2^{(22)},t\right)\bigm|q_1^{(2)}+1=q_2^{(2)}-1=p_1^{(22)}=p_2^{(22)}=0\right\},\\
	&P_3^{(22)} = \left\{\left(q_1^{(2)},q_2^{(2)},p_1^{(22)},p_2^{(22)},t\right)\bigm|q_1^{(2)}-1=q_2^{(2)}=p_1^{(22)}=p_2^{(22)}=0\right\},\\
	&P_4^{(22)} = \left\{\left(q_1^{(2)},q_2^{(2)},p_1^{(22)},p_2^{(22)},t\right)\bigm|q_1^{(2)}-1=q_2^{(2)}=(t-1)p_1^{(22)}-2=p_2^{(22)}=0\right\},\\
	&P_5^{(22)} = \left\{\left(q_1^{(2)},q_2^{(2)},p_1^{(22)},p_2^{(22)},t\right)\bigm|q_1^{(2)}=tq_2^{(2)}-1=2tp_1^{(22)}+t-1=p_2^{(22)}=0\right\},\\
	&P_6^{(22)} = \left\{\left(q_1^{(2)},q_2^{(2)},p_1^{(22)},p_2^{(22)},t\right)\bigm|q_1^{(2)}-1=tq_2^{(2)}-1=tp_1^{(22)}+1=p_2^{(22)}=0\right\},\\
	&P_7^{(22)} = \left\{\left(q_1^{(2)},q_2^{(2)},p_1^{(22)},p_2^{(22)},t\right)\bigm|q_1^{(2)}+1=tq_2^{(2)}+1=tp_1^{(22)}+1=p_2^{(22)}=0\right\}
\end{align*}
on $W_{22}\times B$.
\end{lem}

This lemma will be shown in Appendix \ref{sec:appendix}.
In the following, we use a notation of the singularities
\[
	C_1^{(01)} = \left\{q_1-q_2=p_1^{(01)}=p_2^{(01)}+1=0\right\}
\]
for the sake of simplicity.

Although the number of the accessible singularities $C_j^{(kl)}$ and $P_j^{(kl)}$ are 18 and 42 respectively, we can decrease the number by using relations \eqref{eq:rel_q} and \eqref{eq:rel_p}.

\begin{lem}\label{lem:rel_C}
We have 6 equalities between 2 accessible singularities as
\[
	C_1^{(01)} = C_1^{(02)},\quad
	C_1^{(11)} = C_1^{(12)},\quad
	C_1^{(21)} = C_1^{(22)},\quad
	C_2^{(01)} = C_2^{(11)},\quad
	C_2^{(02)} = C_2^{(22)},\quad
	C_2^{(12)} = C_2^{(21)}.
\]
\end{lem}

\begin{proof}
We first prove the equality $C_1^{(01)}=C_1^{(02)}$.
Relation \eqref{eq:rel_p} implies
\[
	\left(p_1^{(01)},p_2^{(01)}\right) = \left(\frac{p_2^{(02)}}{p_1^{(02)}},\frac{1}{p_1^{(02)}}\right),
\]
from which we obtain that two conditions $p_1^{(01)}=p_2^{(01)}+1=0$ of $C_1^{(01)}$ and $p_1^{(02)}+1=p_2^{(02)}=0$ of $C_1^{(02)}$ are equivalent.
The equalities $C_1^{(11)}=C_1^{(12)}$ and $C_1^{(21)}=C_1^{(22)}$ can be shown in a similar manner.

We next prove the equality $C_2^{(01)}=C_2^{(11)}$.
Relations \eqref{eq:rel_q} and \eqref{eq:rel_p} imply 
\begin{equation}\label{eq:rel_01_10}\begin{split}
	\left(q_1,q_2,p_1^{(01)},p_2^{(01)}\right) = \left(\frac{1}{q_1^{(1)}},\frac{q_2^{(1)}}{q_1^{(1)}},-\frac{p_1^{(11)}}{q_1^{(1)}\left(q_1^{(1)}+q_2^{(1)}p_2^{(11)}+\eta p_1^{(11)}\right)},-\frac{p_2^{(11)}}{q_1^{(1)}+q_2^{(1)}p_2^{(11)}+\eta p_1^{(11)}}\right),
\end{split}\end{equation}
from which we obtain that two conditions $q_1-1=p_1^{(01)}=p_2^{(01)}=0$ of $C_2^{(01)}$ and $q_1^{(1)}-1=p_1^{(11)}=p_2^{(11)}=0$ of $C_2^{(11)}$ are equivalent.
The equalities $C_2^{(02)}=C_2^{(22)}$ and $C_2^{(12)}=C_2^{(21)}$ can be shown in a similar manner.
\end{proof}

\begin{lem}\label{lem:rel_P}
We have 30 equalities between 2 or 4 accessible singularities as
\begin{align*}
	&P_3^{(01)} = P_3^{(21)},\quad
	P_3^{(02)} = P_3^{(12)},\quad
	P_3^{(11)} = P_3^{(22)},\\
	&P_1^{(01)} = P_1^{(11)} = P_6^{(21)} = P_6^{(22)},\quad
	P_2^{(01)} = P_2^{(11)} = P_7^{(21)} = P_7^{(22)},\quad
	P_1^{(02)} = P_6^{(11)} = P_6^{(12)} = P_1^{(22)},\\
	&P_2^{(02)} = P_7^{(11)} = P_7^{(12)} = P_2^{(22)},\quad
	P_4^{(01)} = P_4^{(02)} = P_5^{(21)} = P_5^{(22)},\quad
	P_5^{(01)} = P_5^{(02)} = P_4^{(11)} = P_4^{(12)},\\
	&P_6^{(01)}= P_6^{(02)} = P_1^{(12)} = P_1^{(21)},\quad
	P_7^{(01)} = P_7^{(02)} = P_2^{(12)} = P_2^{(21)},\quad
	P_5^{(11)} = P_5^{(12)} = P_4^{(21)} = P_4^{(22)}.
\end{align*}
\end{lem}

\begin{proof}
We prove only the equality $P_1^{(01)}=P_1^{(11)}=P_6^{(21)}=P_6^{(22)}$ here.
We first obtain from relation \eqref{eq:rel_01_10} that two conditions $q_1-t=q_2-t=p_1^{(01)}=p_2^{(01)}=0$ of $P_1^{(01)}$ and $tq_1^{(1)}-1=q_2^{(1)}-1=p_1^{(11)}=p_2^{(11)}=0$ of $P_1^{(11)}$ are equivalent.
Next, relations \eqref{eq:rel_q} and \eqref{eq:rel_p} imply
\[
	\left(q_1,q_2,p_1^{(01)},p_2^{(01)}\right) = \left(\frac{q_1^{(2)}}{q_2^{(2)}},\frac{1}{q_2^{(2)}},\frac{p_1^{(21)}}{q_2^{(2)}},-\left(q_1^{(2)}+q_2^{(2)}p_2^{(21)}+\eta p_1^{(21)}\right)\right),
\]
from which we obtain that the condition $q_1^{(2)}-1=tq_2^{(2)}-1=p_1^{(21)}=p_2^{(21)}+t=0$ of $P_6^{(21)}$ coincides with that of $P_1^{(01)}$.
In the last, relation \eqref{eq:rel_p} implies
\[
	\left(p_1^{(21)},p_2^{(21)}\right) = \left(\frac{p_2^{(22)}}{p_1^{(22)}},\frac{1}{p_1^{(22)}}\right),
\]
from which we obtain that the condition $q_1^{(2)}-1= tq_2^{(2)}-1=tp_1^{(22)}+1=p_2^{(22)}=0$ of $P_6^{(22)}$ coincides with that of $P_6^{(21)}$.
Hence we obtain the equality
\[
	P_1^{(01)} = P_1^{(11)} = P_6^{(21)} = P_6^{(22)}.
\]
The other equalities can be shown in a similar manner.
\end{proof}

We have reduced the accessible singularities of $\mathcal{H}$ to
\begin{equation}\label{eq:as_list_tmp}\begin{split}
	&C_1^{(01)},\ C_2^{(01)},\ C_3^{(01)},\ C_3^{(02)},\ 
	C_1^{(11)},\ C_3^{(11)},\ C_2^{(12)},\ C_3^{(12)},\ 
	C_1^{(21)},\ C_3^{(21)},\ C_2^{(22)},\ C_3^{(22)},\\ 
	&P_3^{(01)},\ P_4^{(01)},\ P_6^{(01)},\ P_7^{(01)},\ 
	P_3^{(12)},\ P_4^{(12)},\ P_6^{(12)},\ P_7^{(12)},\ 
	P_3^{(22)},\ P_4^{(22)},\ P_6^{(22)},\ P_7^{(22)}
\end{split}\end{equation}
thanks to Lemma \ref{lem:rel_C} and Lemma \ref{lem:rel_P}.
Let $C^0$ and $P^0$ be unions of the accessible singularities defined by
\[
	C^0 = C_1^{(01)}\cup C_2^{(01)} \cup C_3^{(01)}\cup C_3^{(02)}\cup C_1^{(11)}\cup C_3^{(11)}\cup C_2^{(12)}\cup C_3^{(12)}\cup C_1^{(21)}\cup C_3^{(21)}\cup C_2^{(22)}\cup C_3^{(22)}
\]
and
\[
	P^0 = P_3^{(01)}\cup P_4^{(01)}\cup P_6^{(01)}\cup P_7^{(01)}\cup P_3^{(12)}\cup P_4^{(12)}\cup P_6^{(12)}\cup P_7^{(12)}\cup P_3^{(22)}\cup P_4^{(22)}\cup P_6^{(22)}\cup P_7^{(22)}
\]
respectively.
Then we obtain the following lemma.

\begin{lem}\label{lem:C_connect}
The union $C^0$ is described as a disjoint union of three connected components
\[
	C^0 = C_1^0 \sqcup C_2^0 \sqcup C_3^0,
\]
where
\begin{align*}
	&C_1^0 = C_1^{(01)} \cup C_3^{(01)} \cup C_3^{(02)} \cup C_2^{(12)},\\
	&C_2^0 = C_1^{(11)} \cup C_3^{(11)} \cup C_3^{(12)} \cup C_2^{(22)},\\
	&C_3^0 = C_1^{(21)} \cup C_3^{(21)} \cup C_3^{(22)} \cup C_2^{(01)}.
\end{align*}
Moreover, $C^0$ and $P^0$ do not intersect. 
\end{lem}

This lemma will be shown in Appendix \ref{sec:appendixB}.
Lemma \ref{lem:rel_C}, Lemma \ref{lem:rel_P} and Lemma \ref{lem:C_connect} imply the following proposition.

\begin{prop}\label{prop:resolve_as}
All of the accessible singularities of the system $\mathcal{H}$ is reduced to $C^0$ and $P^0$.
\end{prop}

In the last, we investigate solutions passing through the singularity $D^0\setminus(C^0\cup P^0)$.
We consider a point $\left(q_1,q_2,p_1^{(01)},p_2^{(01)},t\right)=(a_0,b_0,0,d_0,t_0)$ such that $F_1^{(01)}(a_0,b_0,0,d_0,t_0)\neq0$ as an example.
System \eqref{eq:W01_1} is rewritten to a Pfaff system
\begin{align*}
	&F_1^{(01)}dt - t(t-1)p_1^{(01)}dq_1 = 0,\\
	&F_1^{(01)}dq_2 - F_2^{(01)}dq_1 = 0,\\
	&F_1^{(01)}dp_1^{(01)} - p_1^{(01)}G_1^{(01)}dq_1 = 0,\\
	&F_1^{(01)}dp_2^{(01)} - G_2^{(01)}dq_1 = 0.
\end{align*}
Then the solution to this system passing through a point $\left(q_1,q_2,p_1^{(01)},p_2^{(01)},t\right)=(a_0,b_0,0,d_0,t_0)$ satisfies
\[
	p_1^{(01)} = 0,\quad
	t = t_0.
\]
This solution corresponds to a vertical leaf.
We can show that all of the leaves corresponding to the solutions to $\mathcal{H}$ passing through the singularity $D^0\setminus(C^0\cup P^0)$ are vertical leaves in a similar manner.

\section{Resolving the accessible singularities $C^0$}\label{sec:resolve_C}

In the previous section, we have determined the accessible singularities of the system $\mathcal{H}$.
As will be seen below, there are infinitely many solutions passing through each accessible singularity.
In this section, we resolve the accessible singularities $C^0$ by introducing new coordinates which parametrize these solutions.
We give the result of this section in advance.

\begin{prop}\label{prop:resolve_C}
Coordinates resolving the accessible singularities $C^0$ are given by
\begin{align*}
	C_1^{(01)} : 
	&\left(x_1^{[01]},y_1^{[01]},z_1^{[01]},w_1^{[01]}\right) = \left(-\{(q_1-q_2)p_1-\alpha_2\}p_1,q_2,\dfrac{1}{p_1},p_1+p_2\right),\\
	C_3^{(01)} : 
	&\left(x_3^{[01]},y_3^{[01]},z_3^{[01]},w_3^{[01]}\right) = \left(-(q_1p_1+q_2p_2-\alpha_5+\eta)p_1,q_2p_1,\dfrac{1}{p_1},\dfrac{p_2}{p_1}\right),\\
	&\left(\widehat{x}_3^{[01]},\widehat{y}_3^{[01]},\widehat{z}_3^{[01]},\widehat{w}_3^{[01]}\right) = r_2\left(x_3^{[01]},y_3^{[01]},z_3^{[01]},w_3^{[01]}\right),\\
	C_3^{(02)} : 
	&\left(x_3^{[02]},y_3^{[02]},z_3^{[02]},w_3^{[02]}\right) = \left(q_1p_2,-(q_1p_1+q_2p_2-\alpha_5+\eta)p_2,\dfrac{p_1}{p_2},\dfrac{1}{p_2}\right),\\
	&\left(\widehat{x}_3^{[02]},\widehat{y}_3^{[02]},\widehat{z}_3^{[02]},\widehat{w}_3^{[02]}\right) = r_2\left(x_3^{[02]},y_3^{[02]},z_3^{[02]},w_3^{[02]}\right),\\
	C_2^{(12)} : 
	&\left(x_2^{[12]},y_2^{[12]},z_2^{[12]},w_2^{[12]}\right) = \left(q_1^{(1)},-\left\{\left(q_2^{(1)}-1\right)p_2^{(1)}-\alpha_2\right\}p_2^{(1)},p_1^{(1)},\dfrac{1}{p_2^{(1)}}\right)
\end{align*}
for $C_1^0$,
\begin{align*}
	C_1^{(11)} :
	&\left(x_1^{[11]},y_1^{[11]},z_1^{[11]},w_1^{[11]}\right) =\left(-\left\{\left(q_1^{(1)}-\dfrac{q_2^{(1)}}{t}\right)p_1^{(1)}-\alpha_4\right\}p_1^{(1)},q_2^{(1)},\dfrac{1}{p_1^{(1)}},p_2^{(1)}+\dfrac{p_1^{(1)}}{t}\right),\\ 
	C_3^{(11)} : 
	&\left(x_3^{[11]},y_3^{[11]},z_3^{[11]},w_3^{[11]}\right) = \left(-\left(q_1^{(1)}p_1^{(1)}+q_2^{(1)}p_2^{(1)}-\alpha_1+\eta\right)p_1^{(1)},q_2^{(1)}p_1^{(1)},\dfrac{1}{p_1^{(1)}},\dfrac{p_2^{(1)}}{p_1^{(1)}}\right),\\
	&\left(\widehat{x}_3^{[11]},\widehat{y}_3^{[11]},\widehat{z}_3^{[11]},\widehat{w}_3^{[11]}\right) = r_4\left(x_3^{[11]},y_3^{[11]},z_3^{[11]},w_3^{[11]}\right),\\
	C_3^{(12)} : 
	&\left(x_3^{[12]},y_3^{[12]},z_3^{[12]},w_3^{[12]}\right) = \left(q_1^{(1)}p_2^{(1)},-\left(q_1^{(1)}p_1^{(1)}+q_2^{(1)}p_2^{(1)}-\alpha_1+\eta\right)p_2^{(1)},\dfrac{p_1^{(1)}}{p_2^{(1)}},\dfrac{1}{p_2^{(1)}}\right),\\
	&\left(\widehat{x}_3^{[12]},\widehat{y}_3^{[12]},\widehat{z}_3^{[12]},\widehat{w}_3^{[12]}\right) = r_4\left(x_3^{[12]},y_3^{[12]},z_3^{[12]},w_3^{[12]}\right),\\
	C_2^{(22)} : 
	&\left(x_2^{[22]},y_2^{[22]},z_2^{[22]},w_2^{[22]}\right) = \left(q_1^{(2)},-\left\{\left(q_2^{(2)}-\frac{1}{t}\right)p_2^{(2)}-\alpha_4\right\}p_2^{(2)},p_1^{(2)},\frac{1}{p_2^{(2)}}\right)
\end{align*}
for $C_2^0$ and
\begin{align*}	
	C_2^{(01)} : 
	&\left(x_2^{[01]},y_2^{[01]},z_2^{[01]},w_2^{[01]}\right) = \left(-\{(q_1-1)p_1-\alpha_0\}p_1,q_2,\dfrac{1}{p_1},p_2\right),\\
	C_1^{(21)} :
	&\left(x_1^{[21]},y_1^{[21]},z_1^{[21]},w_1^{[21]}\right) = \left(-\left\{\left(q_1^{(2)}-q_2^{(2)}\right)p_1^{(2)}-\alpha_0\right\}p_1^{(2)},q_2^{(2)},\dfrac{1}{p_1^{(2)}},p_1^{(2)}+p_2^{(2)}\right),\\ 
	C_3^{(21)} : 
	&\left(x_3^{[21]},y_3^{[21]},z_3^{[21]},w_3^{[21]}\right) = \left(-\left(q_1^{(2)}p_1^{(2)}+q_2^{(2)}p_2^{(2)}-\alpha_3+\eta\right)p_1^{(2)},q_2^{(2)}p_1^{(2)},\dfrac{1}{p_1^{(2)}},\dfrac{p_2^{(2)}}{p_1^{(2)}}\right),\\
	&\left(\widehat{x}_3^{[21]},\widehat{y}_3^{[21]},\widehat{z}_3^{[21]},\widehat{w}_3^{[21]}\right) = r_0\left(x_3^{[21]},y_3^{[21]},z_3^{[21]},w_3^{[21]}\right),\\
	C_3^{(22)} : 
	&\left(x_3^{[22]},y_3^{[22]},z_3^{[22]},w_3^{[22]}\right) = \left(q_1^{(2)}p_2^{(2)},-\left(q_1^{(2)}p_1^{(2)}+q_2^{(2)}p_2^{(2)}-\alpha_3+\eta\right)p_2^{(2)},\dfrac{p_1^{(2)}}{p_2^{(2)}},\dfrac{1}{p_2^{(2)}}\right),\\
	&\left(\widehat{x}_3^{[22]},\widehat{y}_3^{[22]},\widehat{z}_3^{[22]},\widehat{w}_3^{[22]}\right) = r_0\left(x_3^{[22]},y_3^{[22]},z_3^{[22]},w_3^{[22]}\right)
\end{align*}
for $C_3^0$.
Namely, the system on each coordinate $\left(x_j^{[kl]},y_j^{[kl]},z_j^{[kl]},w_j^{[kl]}\right)$ or $\left(\widehat{x}_j^{[kl]},\widehat{y}_j^{[kl]},\widehat{z}_j^{[kl]},\widehat{w}_j^{[kl]}\right)$ is described as a polynomial Hamiltonian system.
\end{prop}

Here we use the notation 
\begin{align*}
	&r_2\left(x_3^{[01]},y_3^{[01]},z_3^{[01]},w_3^{[01]}\right)\\
	&= \left(r_2\left(-(q_1p_1+q_2p_2-\alpha_5+\eta)p_1\right),r_2(q_2p_1),r_2\left(\dfrac{1}{p_1}\right),r_2\left(\dfrac{p_2}{p_1}\right)\right)\\
	&= \bigg(-\frac{(q_1p_1+q_2p_2-\alpha_5+\eta)\{(q_1-q_2)p_1-\alpha_2\}}{q_1-q_2},\frac{q_2\{(q_1-q_2)p_1-\alpha_2\}}{q_1-q_2},\\
	&\quad\qquad\frac{q_1-q_2}{(q_1-q_2)p_1-\alpha_2},\frac{(q_1-q_2)p_2+\alpha_2}{(q_1-q_2)p_1-\alpha_2}\bigg)
\end{align*}
for the sake of simplicity.
This proposition is shown in two ways.
One is the blowing up method and another is that in terms of formal power series.
Note that the latter method will be used in the next section to resolve $P^0$.
Because we can prove in the same way for all of the accessible singularities, we consider only
\begin{align*}
	C_1^0
	&= C_1^{(01)} \cup C_3^{(01)} \cup C_3^{(02)} \cup C_2^{(12)}\\
	&= \left\{q_1-q_2=p_1^{(01)}=p_2^{(01)}+1=0\right\} \cup \left\{q_1=q_2=p_1^{(01)}=0\right\}\\
	&\quad \cup \left\{q_1=q_2=p_2^{(02)}=0\right\} \cup \left\{q_2^{(1)}-1=p_1^{(12)}=p_2^{(12)}=0\right\}
\end{align*}
here.

\subsection{Blowing up}\label{subsec:blow_C}

Let $\left(u_{1i},u_{2i},v_{1i},v_{2i}\right)$, $\left(u'_{1i},u'_{2i},v'_{1i},v'_{2i}\right)$ and $\left(u''_{1i},u''_{2i},v''_{1i},v''_{2i}\right)\ (i=1,2,3,4)$ be affine coordinates of $\mathbb{C}^4$ defined by
\begin{align*}
	&q_1 - q_2 = u_{11}v_{11},\quad
	q_2 = u_{21},\quad
	p_1^{(01)} = v_{11},\quad 
	p_2^{(01)} + 1 = v_{11}v_{21},\\
	&q_1 - q_2 = u_{11}',\quad
	q_2 = u_{21}',\quad
	p_1^{(01)} = u_{11}'v_{11}',\quad 
	p_2^{(01)} + 1 = u_{11}'v_{21}',\\
	&q_1 - q_2 = u_{11}''v_{21}'',\quad
	q_2 = u_{21}'',\quad
	p_1^{(01)} = v_{11}''v_{21}'',\quad 
	p_2^{(01)} + 1 = v_{21}''
\end{align*}
for $C_1^{(01)}$,
\begin{equation}\label{eq:blow_C301}\begin{split}
	&q_1 = u_{12}v_{12},\quad
	q_2 = u_{22}v_{12},\quad
	p_1^{(01)} = v_{12},\quad 
	p_2^{(01)} = v_{22},\\
	&q_1 = u_{12}',\quad
	q_2 = u_{12}'u_{22}',\quad
	p_1^{(01)} = u_{12}'v_{12}',\quad 
	p_2^{(01)} = v_{22}',\\
	&q_1 = u_{12}''u_{22}'',\quad
	q_2 = u_{22}'',\quad
	p_1^{(01)} = u_{22}''v_{12}'',\quad 
	p_2^{(01)} = v_{22}''
\end{split}\end{equation}
for $C_3^{(01)}$,
\begin{align*}
	&q_1 = u_{13}v_{23},\quad
	q_2 = u_{23}v_{23},\quad
	p_1^{(02)} = v_{13},\quad 
	p_2^{(02)} = v_{23},\\
	&q_1 = u_{13}',\quad
	q_2 = u_{13}'u_{23}',\quad
	p_1^{(02)} = v_{13}',\quad 
	p_2^{(02)} = u_{13}'v_{23}',\\
	&q_1 = u_{13}''u_{23}'',\quad
	q_2 = u_{23}'',\quad
	p_1^{(02)} = v_{13}'',\quad 
	p_2^{(02)} = u_{23}''v_{23}''
\end{align*}
for $C_3^{(02)}$ and
\begin{align*}
	&q_1^{(1)} = u_{14},\quad
	q_2^{(1)} - 1 = u_{24}v_{24},\quad
	p_1^{(12)} = v_{14}v_{24},\quad 
	p_2^{(12)} = v_{24},\\
	&q_1^{(1)} = u_{14}',\quad
	q_2^{(1)} - 1 = u_{24}',\quad
	p_1^{(12)} = u_{24}'v_{14}',\quad 
	p_2^{(12)} = u_{24}'v_{24}',\\
	&q_1^{(1)} = u_{14}'',\quad
	q_2^{(1)} - 1 = u_{24}''v_{14}'',\quad
	p_1^{(12)} = v_{14}'',\quad 
	p_2^{(12)} = v_{14}''v_{24}''
\end{align*}
for $C_2^{(12)}$.
On these coordinates, the accessible singularity $C_1^0$ corresponds to
\begin{equation}\label{eq:exdivisor_1}\begin{split}
	D_1^1 
	&= \left\{v_{11}=0\right\} \cup \left\{u_{11}'=0\right\} \cup \left\{v_{21}''=0\right\} \cup \left\{v_{12}=0\right\} \cup \left\{u_{12}'=0\right\} \cup \left\{u_{22}''=0\right\}\\
	&\quad \cup \left\{v_{23}=0\right\} \cup \left\{u_{13}'=0\right\} \cup \left\{u_{23}''=0\right\} \cup \left\{v_{24}=0\right\} \cup \left\{u_{24}'=0\right\} \cup \left\{v_{14}''=0\right\}.
\end{split}\end{equation}
The union $D_1^1$ is connected and called an exceptional divisor.
We consider the accessible singularity $C_3^{(01)}$.
System \eqref{eq:W01_1} on the coordinate $(u_{12},u_{22},v_{12},v_{22})$ is described as 
\begin{equation}\label{eq:blow1st_1}\begin{split}
	&t(t-1)v_{12}du_{12} - U_{12}dt = 0,\\
	&t(t-1)v_{12}du_{22} - U_{22}dt = 0,\\
	&t(t-1)dv_{12} - V_{12}dt = 0,\\
	&t(t-1)dv_{22} - V_{22}dt = 0,
\end{split}\end{equation}
where $U_{12},U_{22},V_{12},V_{22}\in\mathbb{C}[u_{12},u_{22},v_{12},v_{22},t]$.
The singularity and the accessible singularity of system \eqref{eq:blow1st_1} on $D_1^1$ are given by
\[
	\left\{v_{12}=0\right\}
\]
and
\begin{equation}\label{eq:blow1st_1_as}
	\left\{u_{12}+u_{22}v_{22}-\alpha_5+\eta=v_{12}=0\right\}
\end{equation}
respectively.
System \eqref{eq:W01_1} on the coordinate $(u_{12}',u_{22}',v_{12}',v_{22}')$ is described as 
\begin{equation}\label{eq:blow1st_2}\begin{split}
	&t(t-1)v_{12}'du_{12}' - U_{12}'dt = 0,\\
	&t(t-1)u_{12}'v_{12}'du_{22}' - U_{22}'dt = 0,\\
	&t(t-1)u_{12}'dv_{12}' - V_{12}'dt = 0,\\
	&t(t-1)v_{12}'dv_{22}' - V_{22}'dt = 0,
\end{split}\end{equation}
where $U_{12}',U_{22}',V_{12}',V_{22}'\in\mathbb{C}\left[u_{12}',u_{22}',v_{12}',v_{22}',t\right]$.
The singularity and the accessible singularity of system \eqref{eq:blow1st_2} on $D_1^1$ are given by
\[
	\left\{u_{12}'=0\right\} \cup \left\{v_{12}'=0\right\}
\]
and
\begin{equation}\label{eq:blow1st_2_as}
	\left\{u'_{12}=u_{22}'v_{22}'+(-\alpha_5+\eta)v_{12}'+1=0,\ v_{12}'\neq0\right\} \cup \left\{u'_{12}=u_{22}'-1=v'_{12}=v_{22}'+1=0\right\}
\end{equation}
respectively.
System \eqref{eq:W01_1} on the coordinate $(u_{12}'',u_{22}'',v_{12}'',v_{22}'')$ is described as 
\begin{align*}
	&t(t-1)u_{22}''v_{12}''du_{12}'' - U_{12}''dt = 0,\\
	&t(t-1)v_{12}''du_{22}'' - U_{22}''dt = 0,\\
	&t(t-1)u_{22}''dv_{12}'' - V_{12}''dt = 0,\\
	&t(t-1)v_{12}''dv_{22}'' - V_{22}''dt = 0,
\end{align*}
where $U_{12}'',U_{22}'',V_{12}'',V_{22}''\in\mathbb{C}\left[u''_{12},u''_{22},v''_{12},v''_{22},t\right]$.
The singularity and the accessible singularity of this system on $D_1^1$ are given by
\[
	\left\{u_{22}''=0\right\} \cup \left\{v_{12}''=0\right\}
\]
and
\begin{equation}\label{eq:blow1st_3_as}
	\left\{u''_{12}+(-\alpha_5+\eta)v_{12}''+v_{22}''=u_{22}''=0,\ v''_{12}\neq0\right\} \cup \left\{u''_{12}-1=u''_{22}=v_{12}''=v_{22}''+1=0\right\}
\end{equation}
respectively.
Relation \eqref{eq:blow_C301} implies
\begin{align*}
	&\left(u'_{12},u'_{22},v'_{12},v'_{22}\right) = \left(u_{12}v_{12},\frac{u_{22}}{u_{12}},\frac{1}{u_{12}},v_{22}\right),\\
	&\left(u''_{12},u''_{22},v''_{12},v''_{22}\right) = \left(\frac{u_{12}}{u_{22}},u_{22}v_{12},\frac{1}{u_{22}},v_{22}\right),\\
	&\left(u'_{12},u'_{22},v'_{12},v'_{22}\right) = \left(u_{12}''u_{22}'',\frac{1}{u_{12}''},\frac{v_{12}''}{u_{11}''},v_{22}''\right),
\end{align*}
from which the first components \eqref{eq:blow1st_2_as} and \eqref{eq:blow1st_3_as} are both contained in \eqref{eq:blow1st_1_as} and the second components \eqref{eq:blow1st_2_as} and \eqref{eq:blow1st_3_as} are equal.
We can consider the other accessible singularities $C_1^{(01)}$, $C_3^{(02)}$ and $C_2^{(12)}$ in a similar manner.
Consequently, we obtain new accessible singularities
\begin{equation}\label{eq:as_2nd}\begin{split}
	C_1^1 
	&= \left\{u_{11}-\alpha_2=v_{11}=0\right\} \cup \left\{u_{12}+u_{22}v_{22}-\alpha_5+\eta=v_{12}=0\right\}\\
	&\quad \cup \left\{u_{13}v_{13}+u_{23}-\alpha_5+\eta=v_{23}=0\right\} \cup \left\{u_{24}-\alpha_2=v_{24}=0\right\},\\
	P_1^1 
	&= \left\{u'_{12}=u_{22}'-1=v'_{12}=v_{22}'+1=0\right\}
\end{split}\end{equation}
on $D_1^1$.
We extend the system $\mathcal{H}$ to that on $\left(\Sigma_\eta\cup D_1^1\right)\times B$ and denote it by $\mathcal{H}_1^1$ temporarily.
Then there exist transversal leaves corresponding to the solutions to $\mathcal{H}_1^1$ passing through $C_1^1\cup P_1^1$.
Moreover, all of the leaves corresponding to the solutions to $\mathcal{H}_1^1$ passing through the singularity $D_1^1\setminus\left(C_1^1\cup P_1^1\right)$ are vertical leaves.
This fact can be shown in a similar manner as in $D^0\setminus(C^0\cup P^0)$.
We do not state its detail here.

We next consider the blowing up at the accessible singularity $C_1^1$.
We set eight affine coordinates of $\mathbb{C}^4$ as  
\begin{align*}
	&\alpha_2 - u_{11} = x_1^{[01]}z_1^{[01]},\quad
	u_{21} = y_1^{[01]},\quad
	v_{11} = z_1^{[01]},\quad
	v_{21} = w_1^{[01]},\\
	&\alpha_2 - u_{11} = x'_1,\quad
	u_{21} = y'_1,\quad
	v_{11} = x'_1z'_1,\quad
	v_{21} = w'_1,\\
	&-(u_{12}+u_{22}v_{22}-\alpha_5+\eta) = x_3^{[01]}z_3^{[01]},\quad
	u_{22} = y_3^{[01]},\quad
	v_{12} = z_3^{[01]},\quad
	v_{22} = w_3^{[01]},\\
	&-(u_{12}+u_{22}v_{22}-\alpha_5+\eta) = x_{31}',\quad
	u_{22} = y_{31}',\quad
	v_{12} = x_{31}'z_{31}',\quad
	v_{22} = w_{31}',\\
	&u_{13} = x_3^{[02]},\quad
	-(u_{13}v_{13}+u_{23}-\alpha_5+\eta) = y_3^{[02]}w_3^{[02]},\quad
	v_{13} = z_3^{[02]},\quad
	v_{23} = w_3^{[02]},\\
	&u_{13} = x_{32}',\quad
	-(u_{13}v_{13}+u_{23}-\alpha_5+\eta) = y_{32}',\quad
	v_{13} = z_{32}',\quad
	v_{23} = y_{32}'w_{32}',\\
	&u_{14} = x_2^{[12]},\quad
	\alpha_2 - u_{24} = y_2^{[12]}w_2^{[12]},\quad
	v_{14} = z_2^{[12]},\quad
	v_{24} = w_2^{[12]},\\
	&u_{14} = x_2',\quad
	\alpha_2 - u_{24} = y_2',\quad
	v_{14} = z_2',\quad
	v_{24} = y_2'w_2'.
\end{align*}
On these coordinates, the accessible singularity $C_1^1$ corresponds to
\begin{equation}\label{eq:exdivisor_2}\begin{split}
	D_1^2 
	&= \left\{z_1^{[01]}=0\right\} \cup \left\{x'_1=0\right\} \cup \left\{z_3^{[01]}=0\right\} \cup \left\{x'_{31}=0\right\}\\
	&\quad \cup \left\{w_3^{[02]}=0\right\} \cup \left\{y'_{32}=0\right\} \cup \left\{w_2^{[12]}=0\right\} \cup \left\{y'_2=0\right\}.
\end{split}\end{equation}
The union $D_1^2$ is connected and called an exceptional divisor.
We consider the accessible singularity $\left\{u_{12}+u_{22}v_{22}-\alpha_5+\eta=v_{12}=0\right\}$.
System \eqref{eq:blow1st_1} on the coordinate $\left(x_3^{[01]},y_3^{[01]},z_3^{[01]},w_3^{[01]}\right)$ is described as
\begin{equation}\label{eq:blow2nd_1}\begin{split}
	&t(t-1)dx_3^{[01]} - X_3dt = 0,\\
	&t(t-1)dy_3^{[01]} - Y_3dt = 0,\\
	&t(t-1)dz_3^{[01]} - Z_3dt = 0,\\
	&t(t-1)dw_3^{[01]} - W_3dt = 0,
\end{split}\end{equation}
where $X_3,Y_3,Z_3,W_3\in\mathbb{C}\left[x_3^{[01]},y_3^{[01]},z_3^{[01]},w_3^{[01]},t\right]$, which does not have any singularity on $D_1^{2}$.
Then the relation between the coordinates $\left(q_1,q_2,p_1^{(01)},p_2^{(01)}\right)$ and $\left(x_3^{[01]},y_3^{[01]},z_3^{[01]},w_3^{[01]}\right)$ is given by
\begin{equation}\label{eq:C301_coor_01}
	q_1 = -\left(x_3^{[01]}z_3^{[01]}+y_3^{[01]}w_3^{[01]}-\alpha_5+\eta\right)z_3^{[01]},\quad
	q_2 = y_3^{[01]}z_3^{[01]},\quad
	p_1^{(01)} = z_3^{[01]},\quad
	p_2^{(01)} = w_3^{[01]}.
\end{equation}
Note that the accessible singularity $C_3^{(01)}$ corresponds to $\left\{z_3^{[01]}=0\right\}$.
Combining relations \eqref{eq:rel_p} and \eqref{eq:C301_coor_01}, we obtain
\begin{equation}\label{eq:C301_coor}
	x_3^{[01]} = -(q_1p_1+q_2p_2-\alpha_5+\eta)p_1,\quad
	y_3^{[01]} = q_2p_1,\quad
	z_3^{[01]} = \dfrac{1}{p_1},\quad
	w_3^{[01]} = \frac{p_2}{p_1}
\end{equation}
and it is a symplectic transformation.
Moreover, system \eqref{eq:blow2nd_1} is described as a polynomial Hamiltonian system.
We can show it by a direct calculation.
Note that all of the leaves corresponding to the solutions to system \eqref{eq:blow2nd_1} passing through the set $\left\{z_3^{[01]}=0\right\}$ are transversal leaves.
Besides, system \eqref{eq:blow1st_1} on the coordinate $\left(x'_{31},y'_{31},z'_{31},w'_{31}\right)$ is described as
\begin{align*}
	&t(t-1)z'_{31}dx'_{31} - X'_{31}dt = 0,\\
	&t(t-1)z'_{31}dy'_{31} - Y'_{31}dt = 0,\\
	&t(t-1)dz'_{31} - Z'_{31}dt = 0,\\
	&t(t-1)dw'_{31} - W'_{31}dt = 0,
\end{align*}
where $X'_{31},Y'_{31},Z'_{31},W'_{31}\in\mathbb{C}\left[x'_{31},y'_{31},z'_{31},w'_{31},t\right]$.
The singularity of this system on $D_1^2$ is given by
\[
	\left\{x'_{31}=z'_{31}=0\right\}.
\]
We extend the system $\mathcal{H}$ to that on $\left(\Sigma_\eta\cup D_1^1\cup D_1^2\right)\times B$ and denote it by $\mathcal{H}_1^2$ temporarily.
Then all of the leaves corresponding to the solutions to $\mathcal{H}_1^2$ passing through the singularity $\left\{x'_{31}=z'_{31}=0\right\}$ are vertical leaves.
This fact can be shown in a similar manner as in $D^0\setminus(C^0\cup P^0)$.
We can consider the other accessible singularities similarly and obtain
\begin{align*}
	&x_1^{[01]} = -\{(q_1-q_2)p_1-\alpha_2\}p_1,\quad
	y_1^{[01]} = q_2,\quad
	z_1^{[01]} = \dfrac{1}{p_1},\quad
	w_1^{[01]} = p_1 + p_2,\\
	&x_3^{[02]} = q_1p_2,\quad
	y_3^{[02]} = -(q_1p_1+q_2p_2-\alpha_5+\eta)p_2,\quad
	z_3^{[02]} = \dfrac{p_1}{p_2},\quad
	w_3^{[02]} = \dfrac{1}{p_2},\\
	&x_2^{[12]} = q_1^{(1)},\quad
	y_2^{[12]} = -\left\{\left(q_2^{(1)}-1\right)p_2^{(1)}-\alpha_2\right\}p_2^{(1)},\quad
	z_2^{[12]} = p_1^{(1)},\quad
	w_2^{[12]} = \dfrac{1}{p_2^{(1)}}.
\end{align*}
%
%
In the last, we consider the accessible singularity $P_1^1$.
There exist two transversal leaves corresponding to solutions to system \eqref{eq:blow1st_2} passing through $P_1^1$ at any $t=t_0$.
They are expressed in terms of formal power series as
\begin{align*}
	&u_{12}' = \frac{\alpha_2-\alpha_5+\eta}{t_0-1}(t-t_0) + a_2(t-t_0)^2 + O\left((t-t_0)^3\right),\\
	&u_{22}' = 1 - \frac{\alpha_2}{2t_0} + b_2(t-t_0)^2 + O\left((t-t_0)^3\right),\\
	&v_{12}' = \frac{1}{2t_0}(t-t_0) + c_2(t-t_0)^2 + O\left((t-t_0)^3\right),\\
	&v_{22}' = -1 + d_2(t-t_0)^2 + O\left((t-t_0)^3\right)
\end{align*}
and
\begin{align*}
	&u_{12}' = -\frac{\alpha_2-\alpha_5+\eta}{t_0-1}(t-t_0) + a_2(t-t_0)^2 + O\left((t-t_0)^3\right),\\
	&u_{22}' = 1 - \frac{\alpha_2}{t_0} + b_2(t-t_0)^2 + O\left((t-t_0)^3\right),\\
	&v_{12}' = \frac{1}{t_0}(t-t_0) + c_2(t-t_0)^2 + O\left((t-t_0)^3\right),\\
	&v_{22}' = -1 - \frac{\alpha_2-\alpha_5+\eta}{t_0}(t-t_0) + d_2(t-t_0)^2 + O\left((t-t_0)^3\right).
\end{align*}
Here we denote the function $f(t)$ satisfying $\displaystyle\lim_{t\to t_0}\left|\frac{f(t)}{(t-t_0)^k}\right|<\infty$ by $O\left((t-t_0)^k\right)$.
Two coefficients among $a_2,b_2,c_2,d_2$ are not determined and turn to be arbitrary complex numbers.
Since the former solutions passes through the accessible singularity $C_1^{(01)}$, we have to give coordinates which separate the latter solutions.
By using the method which will be given in Section \ref{sec:resolve_P}, we obtain two coordinates 
\begin{align*}
	&\left(\widehat{x}_3^{[01]},\widehat{y}_3^{[01]},\widehat{z}_3^{[01]},\widehat{w}_3^{[01]}\right) = r_2\left(x_3^{[01]},y_3^{[01]},z_3^{[01]},w_3^{[01]}\right),\\
	&\left(\widehat{x}_3^{[02]},\widehat{y}_3^{[02]},\widehat{z}_3^{[02]},\widehat{w}_3^{[02]}\right) = r_2\left(x_3^{[02]},y_3^{[02]},z_3^{[02]},w_3^{[02]}\right).
\end{align*}
We do not state its detail here.

\subsection{Formal power series}\label{subsec:power_C}

We consider $C_2^{(12)}$ of $C_1^0$ in this subsection.
We can prove for $C_1^{(01)}$, $C_3^{(01)}$ and $C_3^{(02)}$ in a similar manner.
The system $\mathcal{H}$ on $W_{12}\times B$ is described as
\begin{equation}\label{eq:W12}\begin{split}
	&t(t-1)p_2^{(12)}dq_1^{(1)} - F_1^{(12)}dt = 0,\\
	&t(t-1)p_2^{(12)}dq_2^{(1)} - F_2^{(12)}dt = 0,\\
	&t(t-1)p_2^{(12)}dp_1^{(12)} - G_1^{(12)}dt = 0,\\
	&t(t-1)dp_2^{(12)} - G_2^{(12)}dt = 0,
\end{split}\end{equation}
where
\begin{align*}
	F_1^{(12)}
	&= \left[(\alpha_5+\eta)t\left(q_1^{(1)}\right)^2+\alpha_3q_1^{(1)}q_2^{(1)}-\left\{(\alpha_0+\alpha_5+\eta)t-\alpha_0-\alpha_1+\eta\right\}q_1^{(1)}-\alpha_3q_2^{(1)}-\alpha _1+\eta\right]p_2^{(12)}\\
	&\quad +2q_1^{(1)}\left(q_1^{(1)}-1\right)\left(tq_1^{(1)}-1\right)p_1^{(12)} + \left(q_1^{(1)}-1\right)\left(tq_1^{(1)}+q_2^{(1)}\right)\left(q_2^{(1)}-1\right),\\
	F_2^{(12)}
	&= \left[\alpha_5tq_1^{(1)}q_2^{(1)}+(\alpha_3+\eta)\left(q_2^{(1)}\right)^2-\alpha_5tq_1^{(1)}+\{(\alpha_1+\alpha_2-\eta)t-\alpha_2-\alpha_3-\eta\}q_2^{(1)}-(\alpha_1-\eta)t\right]p_2^{(12)}\\
	&\quad +\left(q_1^{(1)}-1\right)\left(t q_1^{(1)}+q_2^{(1)}\right)\left(q_2^{(1)}-1\right)p_1^{(12)} + 2q_2^{(1)}\left(q_2^{(1)}-1\right)\left(q_2^{(1)}-t\right),\\
	G_1^{(12)}
	&= \eta\left(\alpha_3p_1^{(12)}-\alpha_5t\right)\left(p_2^{(12)}\right)^2 + \Big[\alpha_3\left(q_1^{(1)}-1\right)\left(p_1^{(12)}\right)^2\\
	&\quad +\left\{(\alpha_3+2\eta)q_2^{(1)}-(\alpha_5+2\eta)tq_1^{(1)}+(-\alpha _3-\alpha_4+1)t+\alpha _4+\alpha_5-1\right\}p_1^{(12)}-\alpha_5tq_2^{(1)}+\alpha_5t\Big]p_2^{(12)}\\
	&\quad -\left\{2t\left(q_1^{(1)}\right)^2-2q_1^{(1)}q_2^{(1)}-(t+1)q_1^{(1)}+2q_2^{(1)}\right\}\left(p_1^{(12)}\right)^2\\
	&\quad -\left\{2tq_1^{(1)}q_2^{(1)}-2\left(q_2^{(1)}\right)^2-2tq_1^{(1)}+(t+1)q_2^{(1)}\right\}p_1^{(12)},\\
	G_2^{(12)}
	&= \alpha_3\eta\left(p_2^{(12)}\right)^2\\
	&\quad + \left[\alpha_3\left(q_1^{(1)}-1\right)p_1^{(12)}+\alpha_5tq_1^{(1)}+2(\alpha _3+\eta)q_2^{(1)}+\left(\alpha_1+\alpha_2-\eta\right)t-\alpha_2-\alpha_3-\eta\right]p_2^{(12)}\\
	&\quad +\left(q_1^{(1)}-1\right)\left(tq_1^{(1)}+2q_2^{(1)}-1\right)p_1^{(12)}+3\left(q_2^{(1)}\right)^2 -2(t+1)q_2^{(1)} + t.
\end{align*}

We fix a point $t_0\in B$ and consider the solution to system \eqref{eq:W12} with the initial condition
\[
	\left(q_1^{(1)}(t_0),q_2^{(1)}(t_0),p_1^{(12)}(t_0),p_2^{(12)}(t_0)\right) = (a_0,1,0,0),
\]
where $a_0\in\mathbb{C}$.
This solution is expressed in terms of formal power series as
\begin{equation}\label{eq:C_power}\begin{split}
	&q_1^{(1)} = a_0 + (t-t_0)^{\lambda_a}\sum_{k=1}^{\infty}a_k(t-t_0)^k,\quad
	q_2^{(1)} = 1 + (t-t_0)^{\lambda_b}\sum_{k=1}^{\infty}b_k(t-t_0)^k,\\
	&p_1^{(12)} = (t-t_0)^{\lambda_c}\sum_{k=1}^{\infty}c_k(t-t_0)^k,\quad
	p_2^{(12)} = (t-t_0)^{\lambda_d}\sum_{k=1}^{\infty}d_k(t-t_0)^k,
\end{split}\end{equation}
where $\lambda_a,a_k,\lambda_b,b_k,\lambda_c,c_k,\lambda_d,d_k\in\mathbb{C}$ and
\[
	{\rm{Re}}(\lambda_a) > -1,\quad
	a_1 \neq 0,\quad
	{\rm{Re}}(\lambda_b) > -1,\quad
	b_1 \neq 0,\quad
	{\rm{Re}}(\lambda_c) > -1,\quad
	c_1 \neq 0,\quad
	{\rm{Re}}(\lambda_d) > -1,\quad
	d_1 \neq 0.
\]
Substituting solution \eqref{eq:C_power} to the fourth equation of system \eqref{eq:W12}, we obtain
\[
	t_0(t_0-1)(\lambda_d+1)d_1(t-t_0)^{\lambda_d} + O\left((t-t_0)^{\lambda_d+1}\right) = -t_0 + 1 + o(1).
\]
Here we denote the function $f(t)$ satisfying $\displaystyle\lim_{t\to t_0}\left|\frac{f(t)}{(t-t_0)^k}\right|=0$ by $o\left((t-t_0)^k\right)$.
It follows that
\begin{equation}\label{eq:C_power_coef_c}
	\lambda_d = 0,\quad
	d_1 = -\frac{1}{t_0}.
\end{equation}
Similarly we obtain 
\begin{equation}\label{eq:C_power_coef_a}
	\lambda_b = 0,\quad
	b_1 = -\frac{\alpha_2}{t_0}
\end{equation}
from the second equation of system \eqref{eq:W12},
\begin{equation}\label{eq:C_power_coef_b}
	\lambda_a = 0
\end{equation}
from the first equation and
\begin{equation}\label{eq:C_power_coef_d}
	\lambda_c = 0
\end{equation}
from the third equation.
Under conditions \eqref{eq:C_power_coef_c}, \eqref{eq:C_power_coef_a}, \eqref{eq:C_power_coef_b} and \eqref{eq:C_power_coef_d}, we also obtain
\[
	a_1 = -\frac{2a_0(a_0-1)(t_0a_0-1)}{t_0-1}c_1 + A_1,
\]
where $A_1\in\mathbb{C}(t_0)[\alpha_0,\ldots,\alpha_5,\eta,a_0]$ from the second equation.
Then we see that $c_1$ turns to be an arbitrary complex number and $a_1$ is determined uniquely by $c_1$.
Moreover, in a similar manner, $c_2$ and $d_2$ are determined uniquely as elements of $\mathbb{C}(t_0)[\alpha_0,\ldots,\alpha_5,\eta,a_0,c_1]$ and we obtain
\[
	a_2 = -\frac{(a_0-1)(t_0a_0+1)}{2(t_0-1)}b_2 + A_2,
\]
where $A_2\in\mathbb{C}(t_0)[\alpha_0,\ldots,\alpha_5,\eta,a_0,c_1]$.
Then we see that $b_2$ turns to be an arbitrary complex number and $a_2$ is determined uniquely by $b_2$.

By using solution \eqref{eq:C_power} with \eqref{eq:C_power_coef_c}, \eqref{eq:C_power_coef_a}, \eqref{eq:C_power_coef_b} and \eqref{eq:C_power_coef_d}, we have
\begin{align*}
	\frac{q_2^{(1)}-1}{p_2^{(12)}} 
	&= \frac{-\alpha_2t_0^{-1}(t-t_0)+b_2(t-t_0)^2+O\left((t-t_0)^3\right)}{-t_0^{-1}(t-t_0)+d_2(t-t_0)^2+O\left((t-t_0)^3\right)}\\
	&= \frac{-\alpha_2t_0^{-1}+b_2(t-t_0)+O\left((t-t_0)^2\right)}{-t_0^{-1}+d_2(t-t_0)+O\left((t-t_0)^2\right)}\\
	&= \alpha_2 + \frac{-t_0^{-1}b_2+\alpha_2t_0^{-1}d_2}{t_0^{-2}}(t-t_0) + O\left((t-t_0)^2\right)\\
	&= \alpha_2 + t_0(-b_2+\alpha_2d_2)(t-t_0) + O\left((t-t_0)^2\right),
\end{align*}
from which we obtain
\begin{align*}
	\frac{\alpha_2-\frac{q_2^{(1)}-1}{p_2^{(12)}}}{p_2^{(12)}}
	&= \frac{t_0(b_2-\alpha_2d_2)(t-t_0)+O\left((t-t_0)^2\right)}{-t_0^{-1}(t-t_0)+O\left((t-t_0)^2\right)}\\
	&= \widetilde{b}_0 + O(t-t_0),
\end{align*}
where $\widetilde{b}_0=-t_0^2(b_2-\alpha_2d_2)$.
We also have
\[
	\frac{p_1^{(12)}}{p_2^{(12)}} = \widetilde{c}_0 + O(t-t_0),
\]
where $\widetilde{c}_0=-t_0c_1$ in a similar manner.
Set
\[
	x_2^{[12]} = q_1^{(1)},\quad
	y_2^{[12]} = \frac{\alpha_2-\frac{q_2^{(1)}-1}{p_2^{(12)}}}{p_2^{(12)}},\quad
	z_2^{[12]} = \frac{p_1^{(12)}}{p_2^{(12)}},\quad
	w_2^{[12]} = p_2^{(12)}.
\]
Then we obtain
\[
	\left(x_2^{[12]}(t_0),y_2^{[12]}(t_0),z_2^{[12]}(t_0),w_2^{[12]}(t_0)\right) = \left(a_0,\widetilde{b}_0,\widetilde{c}_0,0\right).
\]

In the last, we consider the case that the coefficients $\widetilde{b}_0$ and $\widetilde{c}_0$ tend to infinity.
Set
\[
	x = q_1^{(1)},\quad
	y = -\frac{\frac{q_2^{(1)}-1}{p_1^{(12)}}-\alpha_2\frac{p_2^{(12)}}{p_1^{(12)}}}{p_2^{(12)}},\quad
	z = p_1^{(12)},\quad
	w = \frac{p_2^{(12)}}{p_1^{(12)}}
\]
and
\[
	x' = q_1^{(1)},\quad
	y' = \alpha_2-\frac{q_2^{(1)}-1}{p_2^{(12)}},\quad
	z' = \frac{p_1^{(12)}}{\alpha_2-\frac{q_2^{(1)}-1}{p_2^{(12)}}},\quad
	w' = \frac{p_2^{(12)}}{\alpha_2-\frac{q_2^{(1)}-1}{p_2^{(12)}}}.
\]
By using solution \eqref{eq:C_power} with \eqref{eq:C_power_coef_c}, \eqref{eq:C_power_coef_a}, \eqref{eq:C_power_coef_b} and \eqref{eq:C_power_coef_d}, we have
\begin{align*}
	&x = a_0 + O(t-t_0),\quad
	y = \frac{\widetilde{b}_0}{\widetilde{c}_0} + O(t-t_0),\\
	&z = O(t-t_0),\quad
	w = \frac{1}{\widetilde{c}_0} + O(t-t_0)
\end{align*}
and
\begin{align*}
	&x' = a_0 + O(t-t_0),\quad
	y' = O(t-t_0),\\
	&z' = \frac{\widetilde{c}_0}{\widetilde{b}_0} + O(t-t_0),\quad
	w' = \frac{1}{\widetilde{b}_0} + O(t-t_0).
\end{align*}
Hence we obtain
\[
	(x(t_0),y(t_0),z(t_0),w(t_0)) = \left(a_0,\frac{\widetilde{b}_0}{\widetilde{c}_0},0,\frac{1}{\widetilde{c}_0}\right)
\]
and
\[
	(x'(t_0),y'(t_0),z'(t_0),w'(t_0)) = \left(a_0,0,\frac{\widetilde{c}_0}{\widetilde{b}_0},\frac{1}{\widetilde{b}_0}\right).
\]
On the other hand, the systems on new coordinates are described as
\begin{align*}
	&t(t-1)wdx - Xdt  = 0,\\
	&t(t-1)wdy - Ydt = 0,\\
	&t(t-1)wdz - Zdt = 0,\\
	&t(t-1)dw - Wdt = 0,
\end{align*}
where $X, Y, Z, W\in\mathbb{C}[x,y,z,w,t]$ and
\begin{align*}
	&t(t-1)w'dx' - X'dt  = 0,\\
	&t(t-1)w'dy' - Y'dt = 0,\\
	&t(t-1)w'dz' - Z'dt = 0,\\
	&t(t-1)dw' - W'dt = 0,
\end{align*}
where $X', Y', Z', W'\in\mathbb{C}\left[x',y',z',w',t\right]$.
Then all of the leaves corresponding to the solutions to these systems with the initial conditions
\[
	\left(x(t_0),y(t_0),z(t_0),w(t_0)\right) = \left(a_0,\widehat{b}_0,0,0\right),
\]
where $\widehat{b}_0\in\mathbb{C}$ and
\[
	\left(x'(t_0),y'(t_0),z'(t_0),w'(t_0)\right) = (a_0,0,\widehat{c}_0,0),
\]
where $\widehat{c}_0\in\mathbb{C}$ are both vertical leaves.
We do not state their detail here. 

\section{Resolving the accessible singularities $P^0$}\label{sec:resolve_P}

It is not suitable to use the blowing up method to resolve the accessible singularity $P^0$, because the explicit formula of the system after blowing up becomes more complicated.
Therefore we use the method in terms of formal power series and the symmetry of system \eqref{eq:FST} in this section.
We give the result of this section in advance.

\begin{prop}\label{prop:resolve_P}
Coordinates resolving the accessible singularity $P^0$ are given as follows.
\begin{align*}
	&P_3^{(01)}: \left(\widetilde{x}_3^{[01]},\widetilde{y}_3^{[01]},\widetilde{z}_3^{[01]},\widetilde{w}_3^{[01]}\right) = r_3\left(x_3^{[01]},y_3^{[01]},z_3^{[01]},w_3^{[01]}\right).\\
	&P_4^{(01)}: \left(\widetilde{x}_4^{[01]},\widetilde{y}_4^{[01]},\widetilde{z}_4^{[01]},\widetilde{w}_4^{[01]}\right) = r_4r_3\left(x_3^{[01]},y_3^{[01]},z_3^{[01]},w_3^{[01]}\right).\\
	&P_6^{(01)}: \left(\widetilde{x}_6^{[01]},\widetilde{y}_6^{[01]},\widetilde{z}_6^{[01]},\widetilde{w}_6^{[01]}\right) = r_4\left(x_2^{[01]},y_2^{[01]},z_2^{[01]},w_2^{[01]}\right).\\
	&P_7^{(01)}: \left(\widetilde{x}_7^{[01]},\widetilde{y}_7^{[01]},\widetilde{z}_7^{[01]},\widetilde{w}_7^{[01]}\right) = r_5r_4\left(x_2^{[01]},y_2^{[01]},z_2^{[01]},w_2^{[01]}\right).\\
	&P_3^{(12)}: \left(\widetilde{x}_3^{[12]},\widetilde{y}_3^{[12]},\widetilde{z}_3^{[12]},\widetilde{w}_3^{[12]}\right) = r_5\left(x_3^{[12]},y_3^{[12]},z_3^{[12]},w_3^{[12]}\right).\\
	&P_4^{(12)}: \left(\widetilde{x}_4^{[12]},\widetilde{y}_4^{[12]},\widetilde{z}_4^{[12]},\widetilde{w}_4^{[12]}\right) = r_0r_5\left(x_3^{[12]},y_3^{[12]},z_3^{[12]},w_3^{[12]}\right).\\
	&P_6^{(12)}: \left(\widetilde{x}_6^{[12]},\widetilde{y}_6^{[12]},\widetilde{z}_6^{[12]},\widetilde{w}_6^{[12]}\right) = r_0\left(x_2^{[12]},y_2^{[12]},z_2^{[12]},w_2^{[12]}\right).\\
	&P_7^{(12)}: \left(\widetilde{x}_7^{[12]},\widetilde{y}_7^{[12]},\widetilde{z}_7^{[12]},\widetilde{w}_7^{[12]}\right) = r_1r_0\left(x_2^{[12]},y_2^{[12]},z_2^{[12]},w_2^{[12]}\right).\\
	&P_3^{(22)}: \left(\widetilde{x}_3^{[22]},\widetilde{y}_3^{[22]},\widetilde{z}_3^{[22]},\widetilde{w}_3^{[22]}\right) = r_1\left(x_3^{[22]},y_3^{[22]},z_3^{[22]},w_3^{[22]}\right).\\
	&P_4^{(22)}: \left(\widetilde{x}_4^{[22]},\widetilde{y}_4^{[22]},\widetilde{z}_4^{[22]},\widetilde{w}_4^{[22]}\right) = r_2r_1\left(x_3^{[22]},y_3^{[22]},z_3^{[22]},w_3^{[22]}\right).\\
	&P_6^{(22)}: \left(\widetilde{x}_6^{[22]},\widetilde{y}_6^{[22]},\widetilde{z}_6^{[22]},\widetilde{w}_6^{[22]}\right) = r_2\left(x_2^{[22]},y_2^{[22]},z_2^{[22]},w_2^{[22]}\right).\\
	&P_7^{(22)}: \left(\widetilde{x}_7^{[22]},\widetilde{y}_7^{[22]},\widetilde{z}_7^{[22]},\widetilde{w}_7^{[22]}\right) = r_3r_2\left(x_2^{[22]},y_2^{[22]},z_2^{[22]},w_2^{[22]}\right).
\end{align*}
Namely, the system on each coordinate $\left(\widetilde{x}_j^{[kl]},\widetilde{y}_j^{[kl]},\widetilde{z}_j^{[kl]},\widetilde{w}_j^{[kl]}\right)$ is described as a polynomial Hamiltonian system.
\end{prop}

\subsection{Proof of Proposition \ref{prop:resolve_P}}\label{subsec:resolve_P}

Because we can prove in the same way for all of the accessible singularities, we consider only $P_3^{(01)}=\left\{q_1=q_2-t=p_1^{(01)}=p_2^{(01)}=0\right\}$ here.
We fix a point $t_0\in B$ and consider the solution to system \eqref{eq:W01_1} with the initial condition 
\begin{equation}\label{eq:initial_P3}
	\left(q_1(t_0),q_2(t_0),p_1^{(01)}(t_0),p_2^{(01)}(t_0)\right) = (0,t_0,0,0).
\end{equation}
This solution is expressed in terms of formal power series as
\begin{equation}\label{eq:P_power}\begin{split}
	&q_1 = (t-t_0)^{\lambda_a}\sum_{k=1}^{\infty}a_k(t-t_0)^k,\quad
	q_2 = t_0 + (t-t_0)^{\lambda_b}\sum_{k=1}^{\infty}b_k(t-t_0)^k,\\
	&p_1^{(01)} = (t-t_0)^{\lambda_c}\sum_{k=1}^{\infty}c_k(t-t_0)^k,\quad
	p_2^{(01)} = (t-t_0)^{\lambda_d}\sum_{k=1}^{\infty}d_k(t-t_0)^k,
\end{split}\end{equation}
where $\lambda_a,a_k,\lambda_b,b_k,\lambda_c,c_k,\lambda_d,d_k\in\mathbb{C}$ and
\[
	{\rm{Re}}(\lambda_a) > -1,\quad
	a_1 \neq 0,\quad
	{\rm{Re}}(\lambda_b) > -1,\quad
	b_1 \neq 0,\quad
	{\rm{Re}}(\lambda_c) > -1,\quad
	c_1 \neq 0,\quad
	{\rm{Re}}(\lambda_d) > -1,\quad
	d_1 \neq 0.
\]
Substituting solution \eqref{eq:P_power} to the third equation of system \eqref{eq:W01_1}, we obtain
\[
	t_0(t_0-1)(\lambda_c+1)c_1(t-t_0)^{\lambda_c} + O\left((t-t_0)^{\lambda_c+1}\right) = t_0 + o(1).
\]
It follows that
\begin{equation}\label{eq:P_power_coef_c}
	\lambda_c = 0,\quad
	c_1 = \frac{1}{t_0-1}.
\end{equation}
Similarly we obtain 
\begin{equation}\label{eq:P_power_coef_d}
	\lambda_d = 0,\quad
	d_1 = \dfrac{-\alpha_3}{t_0(t_0-1)}
\end{equation}
from the fourth equation of system \eqref{eq:W01_1},
\begin{equation}\label{eq:P_power_coef_a}
	\lambda_a = 0,\quad
	a_1 = \dfrac{\alpha_3+\alpha_5-\eta}{t_0-1}
\end{equation}
from the first equation and
\begin{equation}\label{eq:P_power_coef_b}
	\lambda_b = 0,\quad
	b_1 = 1 - \dfrac{\alpha_4}{2}
\end{equation}
from the second equation.
Moreover, in a similar manner, $b_2$ and $c_2$ are determined uniquely as elements of $\mathbb{C}(t_0)[\alpha_0,\ldots,\alpha_5,\eta]$ and 
$a_2$ and $d_2$ turn to be arbitrary complex numbers.

We consider the action of the transformation $r_3$ on the above solution.
Set
\[
	\widetilde{q}_1 = r_3(q_1),\quad
	\widetilde{q}_2 = r_3(q_2),\quad
	\widetilde{p}_1 = r_3(p_1),\quad
	\widetilde{p}_2 = r_3(p_2),\quad
	\widetilde{p}_1^{(01)} = r_3\left(p_1^{(01)}\right),\quad
	\widetilde{p}_2^{(01)} = r_3\left(p_2^{(01)}\right).
\]
Then relation \eqref{eq:rel_p} implies
\begin{equation}\label{eq:tilde_01}\begin{split}
	&\widetilde{q}_1 = q_1,\\
	&\widetilde{q}_2 = q_2 + \frac{\alpha_3}{p_2} = q_2 + \alpha_3\frac{p_1^{(01)}}{p_2^{(01)}},\\
	&\widetilde{p}_1^{(01)} = r_3\left(\frac{1}{p_1}\right) = \frac{1}{p_1} = p_1^{(01)},\\
	&\widetilde{p}_2^{(01)} = r_3\left(\frac{p_2}{p_1}\right) = \frac{p_2}{p_1} = p_2^{(01)}.
\end{split}\end{equation}
Substituting solution \eqref{eq:P_power} with \eqref{eq:P_power_coef_c}, \eqref{eq:P_power_coef_d}, \eqref{eq:P_power_coef_a} and \eqref{eq:P_power_coef_b} to relation \eqref{eq:tilde_01}, we obtain
\[
	\widetilde{q}_1 = O(t-t_0),\quad
	\widetilde{p}_1^{(01)} = O(t-t_0),\quad
	\widetilde{p}_2^{(01)} = O(t-t_0)
\]
and
\begin{align*}
	\widetilde{q}_2
	&= \left(t_0+O(t-t_0)\right) + \alpha_3\frac{c_1(t-t_0)+O\left((t-t_0)^2\right)}{d_1(t-t_0)+O\left((t-t_0)^2\right)}\\
	&= \left(t_0+\alpha_3\frac{c_1}{d_1}\right) + O(t-t_0)\\
	&= O(t-t_0).
\end{align*}
Hence the solution to system \eqref{eq:W01_1} with initial condition \eqref{eq:initial_P3} is transformed to that to the system on the coordinate $\left(\widetilde{q}_1,\widetilde{q}_2,\widetilde{p}_1^{(01)},\widetilde{p}_2^{(01)}\right)$ whose initial condition is given by
\begin{equation}\label{eq:P3_r3}
	\left(\widetilde{q}_1(t_0),\widetilde{q}_2(t_0),\widetilde{p}_1^{(01)}(t_0),\widetilde{p}_2^{(01)}(t_0)\right) = (0,0,0,0).
\end{equation}
Since system \eqref{eq:FST} is invariant under the transformation $r_3$, the system on $\left(\widetilde{q}_1,\widetilde{q}_2,\widetilde{p}_1,\widetilde{p}_2\right)$ is the same as system \eqref{eq:FST} except changing the parameters.
Hence the system on $\left(\widetilde{q}_1,\widetilde{q}_2,\widetilde{p}_1^{(01)},\widetilde{p}_2^{(01)}\right)$ is the same as system \eqref{eq:W01_1} except changing the parameters and has the accessible singularity
\begin{equation}\label{eq:C3_tilde}
	\left\{\widetilde{q}_1=\widetilde{q}_2=\widetilde{p}_1^{(01)}=0\right\}.
\end{equation}
It is obvious that point \eqref{eq:P3_r3} is contained in accessible singularity \eqref{eq:C3_tilde}.
Note that we have
\[
	\left\{\widetilde{q}_1=\widetilde{q}_2=\widetilde{p}_1^{(01)}=0,\ \widetilde{p}_2^{(01)}\neq0\right\} = \left\{q_1=q_2=p_1^{(01)}=0,\ p_2^{(01)}\neq0\right\}.
\]
On the other hand, we have given the coordinate resolving 
\[
	C_3^{(01)} = \left\{q_1=q_2=p_1^{(01)}=0\right\}
\]
in Proposition \ref{prop:resolve_C}.
Hence we can resolve the accessible singularity \eqref{eq:P3_r3} as
\begin{equation}\label{eq:resolve_Ctil301}
	\left(\widetilde{x}_3^{[01]},\widetilde{y}_3^{[01]},\widetilde{z}_3^{[01]},\widetilde{w}_3^{[01]}\right) = \left(-(\widetilde{q}_1\widetilde{p}_1+\widetilde{q}_2\widetilde{p}_2-\alpha_3-\alpha_5+\eta)\widetilde{p}_1,\widetilde{q}_2\widetilde{p}_1,\dfrac{1}{\widetilde{p}_1},\dfrac{\widetilde{p}_2}{\widetilde{p}_1}\right)
\end{equation}
and it is a symplectic transformation.
Note that the system on $\left(\widetilde{x}_3^{[01]},\widetilde{y}_3^{[01]},\widetilde{z}_3^{[01]},\widetilde{w}_3^{[01]}\right)$ is described as a polynomial Hamiltonian system.

In the last, we see the solution to the system on $\left(\widetilde{x}_3^{[01]},\widetilde{y}_3^{[01]},\widetilde{z}_3^{[01]},\widetilde{w}_3^{[01]}\right)$.
Relations \eqref{eq:tilde_01} and \eqref{eq:resolve_Ctil301} imply
\begin{equation}\label{eq:rel_XY_qp}\begin{split}
	\left(\widetilde{x}_3^{[01]},\widetilde{y}_3^{[01]},\widetilde{z}_3^{[01]},\widetilde{w}_3^{[01]}\right) 
	&= \left(\frac{-q_1+(\alpha_5-\eta)p_1^{(01)}-q_2p_2^{(01)}}{\left(p_1^{(01)}\right)^2},\frac{\alpha_3p_1^{(01)}+q_2p_2^{(01)}}{p_1^{(01)}p_2^{(01)}},p_1^{(01)},p_2^{(01)}\right).
\end{split}\end{equation}
Substituting solution \eqref{eq:P_power} with \eqref{eq:P_power_coef_c}, \eqref{eq:P_power_coef_d}, \eqref{eq:P_power_coef_a} and \eqref{eq:P_power_coef_b} to relation \eqref{eq:rel_XY_qp}, we obtain
\begin{align*}
	\widetilde{x}_3^{[01]}
	&= (A_1a_2+A_2d_2+A_3) + O(t-t_0),\\
	\widetilde{y}_3^{[01]}
	&= \frac{A_4d_2+A_5}{\alpha_3} + O(t-t_0),\\
	\widetilde{z}_3^{[01]}
	&= O(t-t_0),\\
	\widetilde{w}_3^{[01]}
	&= O(t-t_0),
\end{align*}
where $A_1,\ldots,A_5\in\mathbb{C}(t_0)[\alpha_0,\ldots,\alpha_5,\eta]$.
Hence the solution to system \eqref{eq:W01_1} with initial condition \eqref{eq:initial_P3} is transformed to that on $\left(\widetilde{x}_3^{[01]},\widetilde{y}_3^{[01]},\widetilde{z}_3^{[01]},\widetilde{w}_3^{[01]}\right) $ whose initial condition is given by
\[
	\left(\widetilde{x}_3^{[01]},\widetilde{y}_3^{[01]},\widetilde{z}_3^{[01]},\widetilde{w}_3^{[01]}\right) = \left(A_1a_2+A_2d_2+A_3,\frac{A_4d_2+A_5}{\alpha_3},0,0\right).
\]
We can see that the arbitrary complex numbers $a_2$ and $d_2$ in solution \eqref{eq:P_power} appear in this initial condition.

\section{Main result}\label{sec:main}

We summarize the facts obtained in the previous sections.
The system $\mathcal{H}$ has the singularity $D^0$ given in Proposition \ref{prop:sing}.
There exists the accessible singularity $C^0\cup P^0$ on $D^0$ given in \eqref{eq:as_list_tmp}.
All of leaves corresponding the solutions passing through the singularity $D^0\setminus\left(C^0\cup P^0\right)$ are vertical leaves.

Consider the first blowing up along the accessible singularity $C^0\cup P^0$ and denote the proper image of $D^0$ by the same symbol.
Then we obtain the exceptional divisor
\begin{align*}
	D^1 
	&= D_1^1 \cup D_2^1 \cup D_3^1 \cup r_0\left(D_1^1\right) \cup r_1r_0\left(D_1^1\right) \cup r_3\left(D_1^1\right) \cup r_4r_3\left(D_1^1\right)\\
	&\quad \cup r_5\left(D_2^1\right) \cup r_0r_5\left(D_2^1\right) \cup r_2\left(D_2^1\right) \cup r_3r_2\left(D_2^1\right) \cup r_4\left(D_3^1\right) \cup r_5r_4\left(D_3^1\right) \cup r_1\left(D_3^1\right) \cup r_2r_1\left(D_3^1\right).
\end{align*}
Here the component $D_1^1$ is given in \eqref{eq:exdivisor_1}.
The components $D_2^1$ and $D_3^1$ are obtained by the blowing up along $C_2^0$ and $C_3^0$ respectively in a similar manner as $D_1^1$.
Recall that the accessible singularities $C_2^0$ and $C_3^0$ are given in Lemma \ref{lem:C_connect}.
The rest 12 components are obtained by the actions of the transformations $r_0,\ldots,r_5$ as is seen in Section \ref{sec:resolve_P}.
The symbol $r_0\left(D_1^1\right)$ stands for a set given as follows.
If we choose a component $\left\{v_{11}=0\right\}$ of $D_1^1$, then we have
\[
	r_0\left(\{v_{11}=0\}\right) = \left\{r_0(v_{11})=0\right\} = \left\{r_0\left(p_1^{(01)}\right)=0\right\}.
\]
There exists the accessible singularity $C^1\cup P^1$ with
\begin{align*}
	C^1 
	&= C_1^1 \cup C_2^1 \cup C_3^1 \cup r_0\left(C_1^1\right) \cup r_1r_0\left(C_1^1\right) \cup r_3\left(C_1^1\right) \cup r_4r_3\left(C_1^1\right)\\
	&\quad \cup r_5\left(C_2^1\right) \cup r_0r_5\left(C_2^1\right) \cup r_2\left(C_2^1\right) \cup r_3r_2\left(C_2^1\right) \cup r_4\left(C_3^1\right) \cup r_5r_4\left(C_3^1\right) \cup r_1\left(C_3^1\right) \cup r_2r_1\left(C_3^1\right)
\end{align*}
and 
\[
	P^1 = P_1^1 \cup P_2^1 \cup P_3^1.
\]
Here the component $C_1^1\cup P_1^1$ is given in \eqref{eq:as_2nd}.
The components $C_2^1\cup P_2^1$ and $C_3^1\cup P_3^1$ are the accessible singularities on $D_2^1$ and $D_3^1$ respectively.
All of leaves corresponding the solutions passing through the singularity $D^1\setminus\left(C^1\cup P^1\right)$ are vertical leaves.

Consider the second blowing up along the accessible singularity $C^1\cup P^1$ and denote the proper image of $D^1$ by the same symbol.
Then we obtain the exceptional divisor
\begin{align*}
	D^2 
	&= D_1^2 \cup r_2\left(D_1^2\right) \cup D_2^2 \cup r_4\left(D_2^2\right) \cup D_3^2 \cup r_0\left(D_3^2\right) \cup r_0\left(D_1^2\right) \cup r_1r_0\left(D_1^2\right) \cup r_3\left(D_1^2\right) \cup r_4r_3\left(D_1^2\right)\\
	&\quad \cup r_5\left(D_2^2\right) \cup r_0r_5\left(D_2^2\right) \cup r_2\left(D_2^2\right) \cup r_3r_2\left(D_2^2\right) \cup r_4\left(D_3^2\right) \cup r_5r_4\left(D_3^2\right) \cup r_1\left(D_3^2\right) \cup r_2r_1\left(D_3^2\right).
\end{align*}
Here the component $D_1^2$ is given in \eqref{eq:exdivisor_2}.
The components $D_2^2$ and $D_3^2$ are obtained by the blowing up along $C_2^1$ and $C_3^1$ respectively in a similar manner as $D_1^2$.
The rest 15 components are obtained by the actions of the transformations $r_0,\ldots,r_5$.
There is no singularity on $D^2\setminus D^1$ and all of leaves corresponding the solutions passing through the singularity $D^2\cap D^1$ are vertical leaves.

\begin{figure}
	\begin{center}
	\begin{picture}(320,60)
		\put(0,7){\small$D^0$}
		\put(24,74){\small$D_1^1$}
		\put(7,55){\small$D_1^2$}
		\put(-10,35){\small$r_2\left(D_1^2\right)$}
		\put(89,74){\small$D_2^1$}
		\put(72,55){\small$D_2^2$}
		\put(57,35){\small$r_4\left(D_2^2\right)$}
		\put(155,74){\small$D_3^1$}
		\put(137,55){\small$D_3^2$}
		\put(122,35){\small$r_0\left(D_3^2\right)$}
		\put(210,74){\small$r_0\big(D_1^1\big)$}
		\put(187,55){\small$r_0\big(D_1^2\big)$}
		\put(310,74){\small$r_2r_1\big(D_3^1\big)$}
		\put(280,55){\small$r_2r_1\big(D_3^2\big)$}
		\put(247,47){\small$\cdots\cdots$}
		\put(15,10){\line(1,0){340}}
		\put(30,-1){\line(0,1){70}}
		\put(15,50){\line(1,0){30}}
		\put(15,30){\line(1,0){30}}
		\put(95,-1){\line(0,1){70}}
		\put(80,50){\line(1,0){30}}
		\put(80,30){\line(1,0){30}}
		\put(160,-1){\line(0,1){70}}
		\put(145,50){\line(1,0){30}}
		\put(145,30){\line(1,0){30}}
		\put(225,-1){\line(0,1){70}}
		\put(209,50){\line(1,0){30}}
		\put(325,-1){\line(0,1){70}}
		\put(310,50){\line(1,0){30}}
		\end{picture}
	\caption{Exceptional divisors}\label{Fig:exdivisor}
	\end{center}
\end{figure}

We illustrate the above blowing up procedure in Figure \ref{Fig:exdivisor}.
Consequently, a compact manifold $\overline{E}$ is derived from that $\Sigma_\eta\times B$.
Then we obtain an initial value space $E=\overline{E}\setminus\left(D^0\cup D^1\right)$ of system \eqref{eq:FST}.

\begin{thm}\label{thm:main}
The manifold $E$ is constructed by glueing 33 coordinates with relations
\begin{align*}
	\left(q_1^{(1)},q_2^{(1)},p_1^{(1)},p_2^{(1)}\right) 
	&= \left(\frac{1}{q_1},\frac{q_2}{q_1},-q_1(q_1p_1+q_2p_2+\eta),q_1p_2\right),\\
	\left(q_1^{(2)},q_2^{(2)},p_1^{(2)},p_2^{(2)}\right) 
	&= \left(\frac{q_1}{q_2},\frac{1}{q_2},q_2p_1,-q_2(q_1p_1+q_2p_2+\eta)\right),\\
	\left(x_1^{[01]},y_1^{[01]},z_1^{[01]},w_1^{[01]}\right) 
	&= \left(-\{(q_1-q_2)p_1-\alpha_2\}p_1,q_2,\dfrac{1}{p_1},p_1+p_2\right),\\
	\left(x_3^{[01]},y_3^{[01]},z_3^{[01]},w_3^{[01]}\right) 
	&= \left(-(q_1p_1+q_2p_2-\alpha_5+\eta)p_1,q_2p_1,\dfrac{1}{p_1},\dfrac{p_2}{p_1}\right),\\
	\left(x_3^{[02]},y_3^{[02]},z_3^{[02]},w_3^{[02]}\right) 
	&= \left(q_1p_2,-(q_1p_1+q_2p_2-\alpha_5+\eta)p_2,\dfrac{p_1}{p_2},\dfrac{1}{p_2}\right),\\
	\left(x_2^{[12]},y_2^{[12]},z_2^{[12]},w_2^{[12]}\right) 
	&= \left(q_1^{(1)},-\left\{\left(q_2^{(1)}-1\right)p_2^{(1)}-\alpha_2\right\}p_2^{(1)},p_1^{(1)},\dfrac{1}{p_2^{(1)}}\right),\\
	\left(x_1^{[11]},y_1^{[11]},z_1^{[11]},w_1^{[11]}\right) 
	&= \left(-\left\{\left(q_1^{(1)}-\dfrac{q_2^{(1)}}{t}\right)p_1^{(1)}-\alpha_4\right\}p_1^{(1)},q_2^{(1)},\dfrac{1}{p_1^{(1)}},p_2^{(1)}+\dfrac{p_1^{(1)}}{t}\right),\\ 
	\left(x_3^{[11]},y_3^{[11]},z_3^{[11]},w_3^{[11]}\right) 
	&= \left(-\left(q_1^{(1)}p_1^{(1)}+q_2^{(1)}p_2^{(1)}-\alpha_1+\eta\right)p_1^{(1)},q_2^{(1)}p_1^{(1)},\dfrac{1}{p_1^{(1)}},\dfrac{p_2^{(1)}}{p_1^{(1)}}\right),\\
	\left(x_3^{[12]},y_3^{[12]},z_3^{[12]},w_3^{[12]}\right) 
	&= \left(q_1^{(1)}p_2^{(1)},-\left(q_1^{(1)}p_1^{(1)}+q_2^{(1)}p_2^{(1)}-\alpha_1+\eta\right)p_2^{(1)},\dfrac{p_1^{(1)}}{p_2^{(1)}},\dfrac{1}{p_2^{(1)}}\right),\\
	\left(x_2^{[22]},y_2^{[22]},z_2^{[22]},w_2^{[22]}\right) 
	&= \left(q_1^{(2)},-\left\{\left(q_2^{(2)}-\frac{1}{t}\right)p_2^{(2)}-\alpha_4\right\}p_2^{(2)},p_1^{(2)},\frac{1}{p_2^{(2)}}\right),\\
	\left(x_2^{[01]},y_2^{[01]},z_2^{[01]},w_2^{[01]}\right) 
	&= \left(-\{(q_1-1)p_1-\alpha_0\}p_1,q_2,\dfrac{1}{p_1},p_2\right),\\
	\left(x_1^{[21]},y_1^{[21]},z_1^{[21]},w_1^{[21]}\right) 
	&= \left(-\left\{\left(q_1^{(2)}-q_2^{(2)}\right)p_1^{(2)}-\alpha_0\right\}p_1^{(2)},q_2^{(2)},\dfrac{1}{p_1^{(2)}},p_1^{(2)}+p_2^{(2)}\right),\\ 
	\left(x_3^{[21]},y_3^{[21]},z_3^{[21]},w_3^{[21]}\right) 
	&= \left(-\left(q_1^{(2)}p_1^{(2)}+q_2^{(2)}p_2^{(2)}-\alpha_3+\eta\right)p_1^{(2)},q_2^{(2)}p_1^{(2)},\dfrac{1}{p_1^{(2)}},\dfrac{p_2^{(2)}}{p_1^{(2)}}\right),\\
	\left(x_3^{[22]},y_3^{[22]},z_3^{[22]},w_3^{[22]}\right) 
	&= \left(q_1^{(2)}p_2^{(2)},-\left(q_1^{(2)}p_1^{(2)}+q_2^{(2)}p_2^{(2)}-\alpha_3+\eta\right)p_2^{(2)},\dfrac{p_1^{(2)}}{p_2^{(2)}},\dfrac{1}{p_2^{(2)}}\right),\\%
	\left(\widehat{x}_3^{[01]},\widehat{y}_3^{[01]},\widehat{z}_3^{[01]},\widehat{w}_3^{[01]}\right) 
	&= r_2\left(x_3^{[01]},y_3^{[01]},z_3^{[01]},w_3^{[01]}\right),\\
	\left(\widehat{x}_3^{[02]},\widehat{y}_3^{[02]},\widehat{z}_3^{[02]},\widehat{w}_3^{[02]}\right) 
	&= r_2\left(x_3^{[02]},y_3^{[02]},z_3^{[02]},w_3^{[02]}\right),\\
	\left(\widehat{x}_3^{[11]},\widehat{y}_3^{[11]},\widehat{z}_3^{[11]},\widehat{w}_3^{[11]}\right)
	&= r_4\left(x_3^{[11]},y_3^{[11]},z_3^{[11]},w_3^{[11]}\right),\\
	\left(\widehat{x}_3^{[12]},\widehat{y}_3^{[12]},\widehat{z}_3^{[12]},\widehat{w}_3^{[12]}\right) 
	&= r_4\left(x_3^{[12]},y_3^{[12]},z_3^{[12]},w_3^{[12]}\right),\\
	\left(\widehat{x}_3^{[21]},\widehat{y}_3^{[21]},\widehat{z}_3^{[21]},\widehat{w}_3^{[21]}\right) 
	&= r_2\left(x_3^{[21]},y_3^{[21]},z_3^{[21]},w_3^{[21]}\right),\\
	\left(\widehat{x}_3^{[22]},\widehat{y}_3^{[22]},\widehat{z}_3^{[22]},\widehat{w}_3^{[22]}\right)
	&= r_2\left(x_3^{[22]},y_3^{[22]},z_3^{[22]},w_3^{[22]}\right),\\%
	\left(\widetilde{x}_3^{[01]},\widetilde{y}_3^{[01]},\widetilde{z}_3^{[01]},\widetilde{w}_3^{[01]}\right) 
	&= r_3\left(x_3^{[01]},y_3^{[01]},z_3^{[01]},w_3^{[01]}\right),\\
	\left(\widetilde{x}_4^{[01]},\widetilde{y}_4^{[01]},\widetilde{z}_4^{[01]},\widetilde{w}_4^{[01]}\right)
	&= r_4r_3\left(x_3^{[01]},y_3^{[01]},z_3^{[01]},w_3^{[01]}\right),\\
	\left(\widetilde{x}_6^{[01]},\widetilde{y}_6^{[01]},\widetilde{z}_6^{[01]},\widetilde{w}_6^{[01]}\right) 
	&= r_4\left(x_2^{[01]},y_2^{[01]},z_2^{[01]},w_2^{[01]}\right),\\
	\left(\widetilde{x}_7^{[01]},\widetilde{y}_7^{[01]},\widetilde{z}_7^{[01]},\widetilde{w}_7^{[01]}\right)
	&= r_5r_4\left(x_2^{[01]},y_2^{[01]},z_2^{[01]},w_2^{[01]}\right),\\
	\left(\widetilde{x}_3^{[12]},\widetilde{y}_3^{[12]},\widetilde{z}_3^{[12]},\widetilde{w}_3^{[12]}\right) 
	&= r_5\left(x_3^{[12]},y_3^{[12]},z_3^{[12]},w_3^{[12]}\right),\\
	\left(\widetilde{x}_4^{[12]},\widetilde{y}_4^{[12]},\widetilde{z}_4^{[12]},\widetilde{w}_4^{[12]}\right) 
	&= r_0r_5\left(x_3^{[12]},y_3^{[12]},z_3^{[12]},w_3^{[12]}\right),\\
	\left(\widetilde{x}_6^{[12]},\widetilde{y}_6^{[12]},\widetilde{z}_6^{[12]},\widetilde{w}_6^{[12]}\right) 
	&= r_0\left(x_2^{[12]},y_2^{[12]},z_2^{[12]},w_2^{[12]}\right),\\
	\left(\widetilde{x}_7^{[12]},\widetilde{y}_7^{[12]},\widetilde{z}_7^{[12]},\widetilde{w}_7^{[12]}\right) 
	&= r_1r_0\left(x_2^{[12]},y_2^{[12]},z_2^{[12]},w_2^{[12]}\right),\\
	\left(\widetilde{x}_3^{[22]},\widetilde{y}_3^{[22]},\widetilde{z}_3^{[22]},\widetilde{w}_3^{[22]}\right) 
	&= r_1\left(x_3^{[22]},y_3^{[22]},z_3^{[22]},w_3^{[22]}\right),\\
	\left(\widetilde{x}_4^{[22]},\widetilde{y}_4^{[22]},\widetilde{z}_4^{[22]},\widetilde{w}_4^{[22]}\right) 
	&= r_2r_1\left(x_3^{[22]},y_3^{[22]},z_3^{[22]},w_3^{[22]}\right),\\
	\left(\widetilde{x}_6^{[22]},\widetilde{y}_6^{[22]},\widetilde{z}_6^{[22]},\widetilde{w}_6^{[22]}\right) 
	&= r_2\left(x_2^{[22]},y_2^{[22]},z_2^{[22]},w_2^{[22]}\right),\\
	\left(\widetilde{x}_7^{[22]},\widetilde{y}_7^{[22]},\widetilde{z}_7^{[22]},\widetilde{w}_7^{[22]}\right) 
	&= r_3r_2\left(x_2^{[22]},y_2^{[22]},z_2^{[22]},w_2^{[22]}\right).
\end{align*}
\end{thm}

This theorem follows from Propositions \ref{prop:resolve_C} and \ref{prop:resolve_P}.

\begin{rem}
As was seen in Section \ref{sec:intro}, Takenawa constructed another IVS of system \eqref{eq:FST} and gave a geometric interpretation in terms of the N\'{e}ron-Severi bilattice in \cite{Tak1}.
We have not clarified a relationship between this previous work and Theorem \ref{thm:main} yet, which is a future problem.
\end{rem}

\begin{rem}
There exists a unique polynomial Hamiltonian system of degree 5 with respect to $q_1,q_2,p_1,p_2$ which is holomorphic on each coordinates $(x_i,y_i,z_i,w_i)\ (i=0,\ldots,5)$.
Moreover, this Hamiltonian system coincides with \eqref{eq:FST}.
This fact was first mentioned by Sasano in \cite{Sas2}.
It is expected that we can show uniqueness in a similar manner as in \cite{ST}.
However, we have not succeeded in the proof yet.
\end{rem}

\appendix

\section{Proof of Lemma \ref{lem:as_list}}\label{sec:appendix}

We prove only for $W_{01}\times B$.
System \eqref{eq:W01_1} implies that the accessible singularities on $W_{01}\times B$ are given by
\[
	p_1^{(01)} = F_1^{(01)} = F_2^{(01)} = G_2^{(01)} = 0.
\]
We set
\begin{align*}
	\widetilde{F}_1^{(01)} 
	&= F_1^{(01)}\big|_{p_1^{(01)}=0}\\
	&= (q_1-1)\left\{2q_1(q_1-t)+(q_1+q_2)(q_2-t)p_2^{(01)}\right\},\\
	\widetilde{F}_2^{(01)} 
	&= F_2^{(01)}\big|_{p_1^{(01)}=0}\\
	&= (q_2-t)\left\{(q_1-1)(q_1+q_2)+2q_2(q_2-1)p_2^{(01)}\right\},\\
	\widetilde{G}_2^{(01)} 
	&= G_2^{(01)}\big|_{p_1^{(01)}=0}\\
	&= \left[\left\{2q_1^2-2q_1q_2-(t+1)q_1+2q_2\right\}+\left\{2q_1q_2-2q_2^2-2tq_1+(t+1)q_2\right\}p_2^{(01)}\right]p_2^{(01)}.
\end{align*}
Then the accessible singularities are given by a system of algebraic equations
\begin{equation}\label{eq:renritsu}
	\widetilde{F}_1^{(01)} = \widetilde{F}_2^{(01)} = \widetilde{G}_2^{(01)} = 0.
\end{equation}
In the following, we solve it for $q_1$, $q_2$ and $p_2^{(01)}$.

We first assume that $q_1-1=0$.
Under this condition, system \eqref{eq:renritsu} is described as
\[
	q_2(q_2-1)(q_2-t)p_2^{(01)} = \left[t-1+\left\{2q_2^2-(t+3)q_2+2t\right\}p_2^{(01)}\right]p_2^{(01)} = 0.
\]
Hence we obtain the solutions of \eqref{eq:renritsu} as
\begin{align}
	&q_1 - 1 = p_2^{(02)} = 0,\label{C2}\\ 
	&q_1 - 1 = q_2 = 2tp_2^{(02)} + t - 1 = 0,\label{P5}\\ 
	&q_1 - 1 = q_2 - 1 = p_2^{(02)} + 1 = 0,\label{C1sub1}\\
	&q_1 - 1 = q_2 - t = tp_2^{(02)} + 1 = 0\label{P6}.
\end{align}

We next assume that $q_1-1\neq0$ and $q_2-t=0$.
Under this condition, the system $\frac{\widetilde{F}_1^{(01)}}{q_1-1}=\widetilde{F}_2^{(01)}=\widetilde{G}_2^{(01)}=0$ is described as
\[
	q_1(q_1-t) = \left[\left\{2q_1^2-(3t+1)q_1+2t\right\}-t(t-1)p_2^{(01)}\right]p_2^{(01)} = 0.
\]
Hence we obtain the solutions of \eqref{eq:renritsu} as
\begin{align}
	&q_1 = q_2 - t = p_2^{(01)} = 0,\label{P3}\\
	&q_1 = q_2 - t = (t-1)p_2^{(01)} - 2 = 0,\label{P4}\\
	&q_1 - t  = q_2 - t = p_2^{(01)} = 0,\label{P1}\\
	&q_1 - t  = q_2 - t = p_2^{(01)} + 1 = 0\label{C1sub2}.
\end{align}

We next assume that $q_1-1\neq0$, $q_2-t\neq0$ and $p_2^{(01)}=0$.
Under this condition, the system $\frac{\widetilde{F}_1^{(01)}}{q_1-1}=\frac{\widetilde{F}_2^{(01)}}{q_2-t}=\widetilde{G}_2^{(01)}=0$ is described as
\[
	q_1(q_1-t) = q_1 + q_2 = 0.
\]
Hence we obtain the solutions of \eqref{eq:renritsu} as
\begin{align}
	&q_1 = q_2 = p_2^{(01)} = 0,\label{C31}\\
	&q_1 - t  = q_2 + t = p_2^{(01)} = 0\label{P2}.
\end{align}

In the last, we assume that $q_1-1\neq0$, $q_2-t\neq0$ and $p_2^{(01)}\neq0$.
Then the system $\frac{\widetilde{F}_1^{(01)}}{q_1-1}=\frac{\widetilde{F}_2^{(01)}}{q_2-t}=\frac{\widetilde{G}_2^{(01)}}{p_2^{(01)}}=0$ is described as
\begin{equation}\label{eq:mat}
	A\begin{pmatrix}1\\ p_2^{(01)}\end{pmatrix} = \begin{pmatrix}0\\ 0\end{pmatrix},\quad
	A = \begin{pmatrix}2q_1(q_1-t)&(q_1+q_2)(q_2-t)\\ (q_1-1)(q_1+q_2)&2q_2(q_2-1)\\ 2q_1^2-2q_1q_2-(t+1)q_1+2q_2&2q_1q_2-2q_2^2-2tq_1+(t+1)q_2\end{pmatrix}.
\end{equation}
It is necessary that all of the minor determinants of the matrix $A$ are zero, namely
\begin{align*}
	(q_1-q_2)X = (q_1-q_2)Y = (q_1-q_2)Z = 0,
\end{align*}
where
\begin{align*}
	&X = q_1^2q_2 - q_1q_2^2 - tq_1^2 - 3(t-1)q_1q_2 + q_2^2 + tq_1 - tq_2 ,\\
	&Y = 2q_1^2q_2 - 2q_1q_2^2 - 2tq_1^2 - 3(t-1)q_1q_2 + 2q_2^2 + 2tq_1 + (t-3)q_2 ,\\
	&Z = -2q_1^2q_2 + 2q_1q_2^2 + 2tq_1^2 + 3(t-1)q_1q_2 - 2q_2^2 - t(3t-1)q_1 + 2tq_2.
\end{align*}
We first consider the case $q_1-q_2=0$.
Then the matrix $A$ is described as 
\[
	A = \begin{pmatrix}2q_1(q_1-t)&2q_1(q_1-t)\\ 2q_1(q_1-1)&2q_1(q_1-1)\\ -(t-1)q_1&-(t-1)q_1\end{pmatrix},
\]
from which we obtain two solutions $q_1=q_2=0$ and $p_2^{(01)}+1=0$ for \eqref{eq:mat}.
We next consider the case $q_1-q_2\neq0$ and $X=Y=Z=0$.
Since $Y+Z=-3(t-1)(tq_1-q_2)=0$, we have
\[
	X = t(t-1)q_1(q_1+1)^2,\quad
	Y = t(t-1)q_1(q_1+1)(2q_1-1),\quad
	Z = -t(t-1)q_1(q_1+1)(2q_1-1).
\]
It implies
\[
	q_1 + 1 = q_2 + t = 0.
\]
Then the matrix $A$ is described as 
\[
	A = \begin{pmatrix}2(t+1)&2t(t+1)\\ 2(t+1)&2t(t+1)\\ -3(t-1)&-3t(t-1)\end{pmatrix},
\]
from which we obtain a solution $tp_2^{(01)}+1=0$ for \eqref{eq:mat}.
Hence we obtain three solutions of \eqref{eq:renritsu} as
\begin{align}
	&q_1 = q_2 = 0,\quad
	p_2^{(01)}\neq0,\label{C32}\\
	&q_1 - q_2 = p_2^{(01)} + 1 = 0,\label{C1}\\
	&q_1 + 1 = q_2 + t = tp_2^{(01)} + 1 = 0.\label{P7}
\end{align}

We have obtained 13 solutions.
Among them, solutions \eqref{C1sub1} and \eqref{C1sub2} are contained in \eqref{C1}.
Moreover, \eqref{C31} and \eqref{C32} can be summarized.
Therefore the accessible singularities $C_1^{(01)},\ldots,C_3^{(01)}$ are derived from \eqref{C1}, \eqref{C2}, \eqref{C31} and \eqref{C32} and $P_1^{(01)},\ldots,P_7^{(01)}$ from \eqref{P1}, \eqref{P2}, \eqref{P3}, \eqref{P4}, \eqref{P5}, \eqref{P6}, \eqref{P7}.

We can prove for the other $W_{ij}\times B$ in a similar manner.

\section{Proof of Lemma \ref{lem:C_connect}}\label{sec:appendixB}
Using the homogeneous coordinates \eqref{eq:hom_affine}, we have
\begin{equation}\label{eq:as_hom}\begin{split}
	C_1^{(01)}
	&= \left\{\left(\xi,\zeta^{(0)},t\right)\bigm|\xi_1-\xi_2=\zeta_0^{(0)}=\zeta_1^{(0)}+\zeta_2^{(0)}=0,\ \xi_0\neq0,\ \zeta_1^{(0)}\neq0\right\}\\
	&= \left\{\left([\xi_0:\xi_1:\xi_1],[0:1:-1],t\right)\bigm|\xi_0\neq0\right\},\\
	C_2^{(01)} 
	&= \left\{\left(\xi,\zeta^{(0)},t\right)\bigm|\xi_0-\xi_1=\zeta_0^{(0)}=\zeta_1^{(0)}=0,\ \xi_0\neq0,\ \zeta_2^{(0)}\neq0\right\}\\
	&= \left\{\left([\xi_0:\xi_0:\xi_2],[0:1:0],t\right)\bigm|\xi_0\neq0\right\},\\
	C_3^{(01)} 
	&= \left\{\left(\xi,\zeta^{(0)},t\right)\bigm|\xi_1=\xi_2=\zeta_0^{(0)}=0,\ \xi_0\neq0,\ \zeta_1^{(0)}\neq0\right\}\\
	&= \left\{\left([1:0:0],\left[0:\zeta_1^{(0)}:\zeta_2^{(0)}\right],t\right)\bigm|\zeta_1^{(0)}\neq0\right\},\\
	C_3^{(02)} 
	&= \left\{\left(\xi,\zeta^{(0)},t\right)\bigm|\xi_1=\xi_2=\zeta_0^{(0)}=0,\ \xi_0\neq0,\ \zeta_2^{(0)}\neq0\right\}\\
	&= \left\{\left([1:0:0],\left[0:\zeta_1^{(0)}:\zeta_2^{(0)}\right],t\right)\bigm|\zeta_2^{(0)}\neq0\right\},\\
	C_1^{(11)} 
	&= \left\{\left(\xi,\zeta^{(1)},t\right)\bigm|t\xi_0-\xi_2=\zeta_0^{(1)}=\zeta_1^{(1)}+t\zeta_2^{(1)}=0,\ \xi_1\neq0,\ \zeta_1^{(1)}\neq0\right\}\\
	&= \left\{\left([\xi_0:\xi_1,t\xi_0],[0:t:-1],t\right)\bigm|\xi_1\neq0\right\},\\
	C_3^{(11)} 
	&= \left\{\left(\xi,\zeta^{(1)},t\right)\bigm|\xi_0=\xi_2=\zeta_0^{(1)}=0,\ \xi_1\neq0,\ \zeta_1^{(1)}\neq0\right\}\\
	&= \left\{\left([0:1:0],\left[0:\zeta_1^{(1)}:\zeta_2^{(1)}\right],t\right)\bigm|\zeta_1^{(1)}\neq0\right\},\\
	C_2^{(12)}
	&= \left\{\left(\xi,\zeta^{(1)},t\right)\bigm|\xi_1-\xi_2=\zeta_0^{(1)}=\zeta_1^{(1)}=0,\ \xi_1\neq0,\ \zeta_2^{(1)}\neq0\right\}\\
	&= \left\{\left([\xi_0:\xi_1:\xi_1],[0:0:1],t\right)\bigm|\xi_1\neq0\right\},\\
	C_3^{(12)} 
	&= \left\{\left(\xi,\zeta^{(1)},t\right)\bigm|\xi_0=\xi_2=\zeta_0^{(1)}=0,\ \xi_1\neq0,\ \zeta_2^{(1)}\neq0\right\}\\
	&= \left\{\left([0:1:0],\left[0:\zeta_1^{(1)}:\zeta_2^{(1)}\right],t\right)\bigm|\zeta_2^{(1)}\neq0\right\},\\
	C_1^{(21)} 
	&= \left\{\left(\xi,\zeta^{(2)},t\right)\bigm|\xi_0-\xi_1=\zeta_0^{(2)}=\zeta_1^{(2)}+\zeta_2^{(2)}=0,\ \xi_2\neq0,\ \zeta_1^{(2)}\neq0\right\}\\
	&= \left\{\left([\xi_0:\xi_0:\xi_2],[0:1:-1],t\right)\bigm|\xi_2\neq0\right\},\\
	C_3^{(21)} 
	&= \left\{\left(\xi,\zeta^{(2)},t\right)\bigm|\xi_0=\xi_1=\zeta_0^{(2)}=0,\ \xi_2\neq0,\ \zeta_1^{(2)}\neq0\right\}\\
	&= \left\{\left([0:0:1],\left[0:\zeta_1^{(2)}:\zeta_2^{(2)}\right],t\right)\bigm|\zeta_1^{(2)}\neq0\right\},\\
	C_2^{(22)}
	&= \left\{\left(\xi,\zeta^{(2)},t\right)\bigm|t\xi_0-\xi_2=\zeta_0^{(2)}=\zeta_1^{(2)}=0,\ \xi_2\neq0,\ \zeta_2^{(2)}\neq0\right\}\\
	&= \left\{\left(\left[t^{-1}\xi_2:\xi_1:\xi_2\right],[0:0:1],t\right)\bigm|\xi_2\neq0\right\},\\
	C_3^{(22)}
	&= \left\{\left(\xi,\zeta^{(2)},t\right)\bigm|\xi_0=\xi_1=\zeta_0^{(2)}=0,\ \xi_2\neq0,\ \zeta_2^{(2)}\neq0\right\}\\
	&= \left\{\left([0:0:1],\left[0:\zeta_1^{(2)}:\zeta_2^{(2)}\right],t\right)\bigm|\zeta_2^{(2)}\neq0\right\}.
\end{split}\end{equation}
It implies
\[
	C_3^{(01)}\cup C_3^{(02)}\simeq\mathbb{P}^1,\quad
	C_3^{(11)}\cup C_3^{(12)}\simeq\mathbb{P}^1,\quad
	C_3^{(21)}\cup C_3^{(22)}\simeq\mathbb{P}^1.
\]
On the other hand, we obtain
\[
	C_1^{(01)} \cup C_2^{(12)} \simeq \mathbb{P}^1,\quad
	C_1^{(11)} \cup C_2^{(22)} \simeq \mathbb{P}^1,\quad
	C_2^{(02)} \cup C_1^{(21)} \simeq \mathbb{P}^1.
\]
The first relation is obtained from that
\begin{equation}\label{eq:C101_W1}\begin{split}
	C_1^{(01)}
	&= \left\{\left(\xi,\zeta^{(1)},t\right)\bigm|\xi_1-\xi_2=\zeta_0^{(1)}=\zeta_1^{(1)}=0,\ \xi_0\neq0,\ \xi_1\neq0,\ \zeta_2^{(1)}\neq0\right\}\\
	&= \left\{\left([\xi_0:\xi_1:\xi_1],[0:0:1],t\right)\bigm|\xi_0\neq0,\ \xi_1\neq0\right\}
\end{split}\end{equation}
under the condition $\xi_1\neq0$.
Note that relation \eqref{eq:C101_W1} is derived from that \eqref{eq:rel_hom}.
We can obtain the other relations in a similar manner.

We next consider the intersections of two of the six unions
\begin{align*}
	&C_3^{(01)} \cup C_3^{(02)},\quad
	C_3^{(11)} \cup C_3^{(12)},\quad
	C_3^{(21)} \cup C_3^{(22)},\\
	&C_1^{(01)} \cup C_2^{(12)},\quad
	C_1^{(11)} \cup C_2^{(22)},\quad
	C_2^{(02)} \cup C_1^{(21)}.
\end{align*}
Definitions of the accessible singularities \eqref{eq:as_hom} imply
\begin{align*}
	&\left(C_3^{(01)}\cup C_3^{(02)}\right) \cap \left(C_1^{(01)}\cup C_2^{(12)}\right)\\
	&= \left\{\left(\xi,\zeta^{(0)},t\right)\bigm|\xi_0=1,\ \xi_1=\xi_2=0,\ \zeta_0^{(0)}=0,\ \zeta_1^{(0)}=1,\ \zeta_2^{(0)}=-1\right\}\\
	&= \{([1:0:0],[0:1:-1],t)\},\\
	&\left(C_3^{(11)}\cup C_3^{(12)}\right) \cap \left(C_1^{(11)}\cup C_2^{(22)}\right)\\
	&= \left\{\left(\xi,\zeta^{(1)},t\right)\bigm|\xi_1=1,\ \xi_0=\xi_2=0,\ \zeta_0^{(1)}=0,\ \zeta_1^{(1)}=-t,\ \zeta_2^{(1)}=1\right\}\\
	&= \{([0:1:0],[0:-t:1],t)\},\\
	&\left(C_3^{(21)}\cup C_3^{(22)}\right) \cap \left(C_1^{(02)}\cup C_1^{(21)}\right)\\
	&= \left\{\left(\xi,\zeta^{(2)},t\right)\bigm|\xi_0=\xi_1=0,\ \xi_2=1,\ \zeta_0^{(2)}=0,\ \zeta_1^{(2)}=1,\ \zeta_2^{(2)}=-1\right\}\\
	&= \{([0:1:0],[0:-t:1],t)\}.
\end{align*}
Moreover, we obtain
\begin{align*}
	&\left(C_3^{(01)}\cup C_3^{(02)}\right) \cap \left(C_3^{(11)}\cup C_3^{(12)}\right) = \emptyset,\quad
	\left(C_3^{(01)}\cup C_3^{(02)}\right) \cap \left(C_3^{(21)}\cup C_3^{(22)}\right) = \emptyset,\\
	&\left(C_3^{(01)}\cup C_3^{(02)}\right) \cap \left(C_1^{(11)}\cup C_2^{(22)}\right) = \emptyset,\quad
	\left(C_3^{(01)}\cup C_3^{(02)}\right) \cap \left(C_2^{(02)}\cup C_1^{(21)}\right) = \emptyset,\\
	&\left(C_1^{(01)}\cup C_2^{(12)}\right) \cap \left(C_3^{(11)}\cup C_3^{(12)}\right) = \emptyset,\quad
	\left(C_1^{(01)}\cup C_2^{(12)}\right) \cap \left(C_3^{(21)}\cup C_3^{(22)}\right) = \emptyset,\\
	&\left(C_1^{(01)}\cup C_2^{(12)}\right) \cap \left(C_1^{(11)}\cup C_2^{(22)}\right) = \emptyset,\quad
	\left(C_1^{(01)}\cup C_2^{(12)}\right) \cap \left(C_2^{(02)}\cup C_1^{(21)}\right) = \emptyset,\\
	&\left(C_3^{(11)}\cup C_3^{(12)}\right) \cap \left(C_3^{(21)}\cup C_3^{(22)}\right) = \emptyset,\quad
	\left(C_3^{(11)}\cup C_3^{(12)}\right) \cap \left(C_2^{(02)}\cup C_1^{(21)}\right) = \emptyset,\\
	&\left(C_1^{(11)}\cup C_2^{(22)}\right) \cap \left(C_3^{(21)}\cup C_3^{(22)}\right) = \emptyset,\quad
	\left(C_1^{(11)}\cup C_2^{(22)}\right) \cap \left(C_2^{(02)}\cup C_1^{(21)}\right) = \emptyset.
\end{align*}
The first relation is obtained from those
\begin{align*}
	&\left(C_1^{(01)}\cup C_2^{(12)}\right) \cap \left(C_1^{(11)}\cup C_2^{(22)}\right)\\
	&= \left(C_1^{(01)} \cap C_1^{(11)}\right) \cup \left(C_1^{(01)} \cap C_2^{(22)}\right) \cup \left(C_2^{(12)} \cap C_1^{(11)}\right) \cup \left(C_2^{(12)} \cap C_2^{(22)}\right)
\end{align*}
and
\[
	C_1^{(01)} \cap C_1^{(11)} = \emptyset,\quad
	C_1^{(01)} \cap C_2^{(22)} = \emptyset,\quad
	C_2^{(12)} \cap C_1^{(11)} = \emptyset,\quad
	C_2^{(12)} \cap C_2^{(22)} = \emptyset.
\]
Note that the relation $C_1^{(01)}\cap C_1^{(11)}=\emptyset$ is derived from that \eqref{eq:C101_W1}.
We can obtain the other relations in a similar manner.

In the last, we consider the accessible singularity $P^0$. 
We choose the point $P_3^{(01)}\subset P^0$ as an example.
It is described as
\begin{align*}
	P_3^{(01)} 
	&= \left\{\left(\xi,\zeta^{(0)},t\right)\bigm|\xi_1=\xi_2-t\xi_0=\zeta_0^{(0)}=\zeta_2^{(0)}=0,\ \xi_0\neq0,\ \zeta_1^{(0)}\neq0\right\}\\
	&=\left\{\left([1:0:t],[0:1:0],t\right)\right\}.
\end{align*}
Recall that the point $P_3^{(01)}$ is equivalent to that
\begin{align*}
	P_3^{(21)} 
	&= \left\{\left(\xi,\zeta^{(2)},t\right)\bigm|\xi_1=\xi_2-t\xi_0=\zeta_0^{(2)}=\zeta_2^{(2)}=0,\ \xi_2\neq0,\ \zeta_1^{(2)}\neq0\right\}\\
	&=\left\{\left([1:0:t],[0:1:0],t\right)\right\}.
\end{align*}
We can see that $P_3^{(01)}=P_3^{(21)}$ is not contained in $C^0$.
We can prove for the other points in a similar manner.

\section*{Acknowledgement}
This work was supported by JSPS KAKENHI Grant Number 20K03645.
The author declares no conflicts of interest associated with this manuscript.

\section*{Data availability}
Data sharing not applicable to this article as no datasets were generated or analyzed during the current study.


\end{document}